\documentclass[12pt,reqno]{amsart}

\usepackage{hyperref}
\hypersetup{colorlinks=true, linkcolor=red, citecolor=blue}

\usepackage{amssymb,amsmath,amsfonts,latexsym,mathrsfs}
\usepackage{bm}
\usepackage{mathtools}
\usepackage{enumerate}
\usepackage{array,graphics,color}
\allowdisplaybreaks
\usepackage{mathrsfs}
\usepackage{mathtools}
\usepackage{blkarray, bigstrut}
\usepackage{csquotes}
\usepackage{ulem}
\usepackage{dsfont}

\newcounter{relctr} 
\everydisplay\expandafter{\the\everydisplay\setcounter{relctr}{0}} 

\newcommand\labelrel[2]{%
  \begingroup
    \refstepcounter{relctr}%
    \stackrel{\textnormal{(\alph{relctr})}}{\mathstrut{#1}}%
    \originallabel{#2}%
  \endgroup
}
\AtBeginDocument{\let\originallabel\label}

\numberwithin{equation}{section}

\newtheorem{dfn}{Definition}[section]
\newtheorem{thm}[dfn]{Theorem}
\newtheorem{lma}[dfn]{Lemma}
\newtheorem{ppsn}[dfn]{Proposition}
\newtheorem{crl}[dfn]{Corollary}

\newtheorem{rmrk}[dfn]{Remark}

\newcommand{\cir}{\mathbb{T}}
\newcommand{\R}{\mathbb{R}}
\newcommand{\C}{\mathbb{C}}
\newcommand{\N}{\mathbb{N}}	

\newcommand{\hilh}{\mathcal{H}}

\newcommand{\bh}{\mathcal{B}(\hilh)}

\newcommand{\Sp}{\mathcal{S}}

\newcommand{\Tr}{\operatorname{Tr}}

\begin{document}
	
\title[Functions of unitaries for non continuously differentiable functions]{Functions of unitaries with $\mathcal{S}^p$-perturbations for non continuously differentiable functions}

\author[C. Coine]{Cl\'ement Coine}
\address{Normandie Univ, UNICAEN, CNRS, LMNO, 14000 Caen, France}	
\email{clement.coine@unicaen.fr}

	

\begin{abstract}
Consider a function $f : \mathbb{T} \to \mathbb{C}$, $n$-times differentiable on $\mathbb{T}$ and such that its $n$th derivative $f^{(n)}$ is bounded but not necessarily continuous. Let $U : \mathbb{R} \to \mathcal{U}(\hilh)$ be a function taking values in the set of unitary operators on some separable Hilbert space $\hilh$. Let $1<p<\infty$ and let $\mathcal{S}^p(\hilh)$ be the Schatten class of order $p$ on $\hilh$. If $\tilde{U}:t\in\R\mapsto U(t)-U(0)$ is $n$-times $\Sp^{p}$-differentiable on $\mathbb{R}$, we show that the operator valued function $\varphi : t\in \mathbb{R} \mapsto f(U(t)) - f(U(0)) \in \mathcal{S}^p(\hilh)$ is $n$-times differentiable on $\mathbb{R}$ as well. This theorem is optimal and extends several results related to the differentiability of functions of unitaries. The derivatives of $\varphi$ are given in terms of multiple operator integrals and a formula and $\mathcal{S}^p$-estimates for the Taylor remainders of $\varphi$ are provided.
\end{abstract}
\maketitle


\section{Introduction}

Let $\hilh$ be a complex separable Hilbert space. Let $\mathcal{B}(\hilh)$ denote the Banach space of bounded operators on $\hilh$, and let $\mathcal{U}(\hilh)$ be the subset of unitary operators. For any $1<p<\infty$, $\mathcal{S}^p(\hilh)$ will denote the Schatten class of order $p$ on $\hilh$, that is the Banach space defined by
\begin{align*}
\Sp^{p}(\hilh)=\left\{A\in\bh \mid \|A\|_{p}:=\Tr(|A|^{p})^{1/p}<\infty\right\}.
\end{align*}
The study of differentiability of operator functions was initiated in \cite{DaKr56}. Since then, it has attracted a lot of attention and significant refinements have been obtained in \cite{ArBaFr90, AzCaDoSu09, BiSo3, DePSu04, KiSh04, Pe00, Pe06, St77}. This study has often been motivated by problems in perturbation theory. For instance, the various and fruitful efforts to prove the existence of spectral shift functions, see \cite{Neid, PoSkSucont,  PoSkSu13, Sk17}, naturally led to the question of the existence and the representation of the derivatives of
$$
\varphi : t\in \mathbb{R} \mapsto f(e^{itA}U) - f(U),
$$
where $A=A^* \in \mathcal{B}(\hilh)$, $U \in \mathcal{U}(\hilh)$ and $f : \mathbb{T} \to \mathbb{C}$ is a function defined on $\mathbb{T}$, the unit circle of $\mathbb{C}$. In \cite{PoSkSu16}, the authors proved that if $f$ belongs to the Besov class $B^n_{\infty 1}(\cir)$, $n\geq 2$, the $n$th order derivative of $\varphi$ exists in the operator norm. For the Schatten classes, it was proved in \cite{CCGP} that if $1<p<\infty$ and $A\in \mathcal{S}^p(\hilh)$, then, under the assumption $f\in C^n(\mathbb{T})$, the function $\varphi$ is $n$-times continuously $\Sp^{p}$-differentiable on $\mathbb{R}$. In fact, stronger results hold, see \cite[Theorem 3.3]{CCGP}. The common denominator in all these results is the use of the theory of multiple operator integrals, which can be seen as the measurable counterparts of Schur multipliers. In particular, the derivatives of $\varphi$ can be expressed as multiple operator integrals or a linear combination of them, with respect to the divided differences of $f$. See for instance \cite[Theorem 3.7]{PoSkSuTo17} for the finite dimensional case and \cite[Theorem 3.5]{CCGP} for the infinite dimensional case.

In the selfadjoint case, more is known. The analogue question is to investigate under which assumptions on $g : \mathbb{R} \to \mathbb{C}$,
the function
$$
\psi : t \in \mathbb{R} \mapsto g(A+tK) - g(A)
$$
is differentiable, where $A$ and $K$ are selfadjoint with $K$ bounded. When $g \in C^n(\mathbb{R})$ with bounded derivatives and $K\in \mathcal{S}^p$ with $1<p<\infty$, it is known that $\psi$ is $\mathcal{S}^p$-differentiable with continuous derivatives, see \cite{CoLeSkSu17, LeSk20}. In fact, the existence of $\psi'$ in the $\mathcal{S}^p$-norm holds when the assumptions on $g$ are relaxed. Indeed, a striking result is given in \cite{KiPoShSu14}, where the authors proved that the condition ``$g$ differentiable on $\mathbb{R}$ with bounded derivative" ensures the differentiability of $\psi$ in the $\mathcal{S}^p$-norm. This is a fundamental difference with the $\mathcal{B}(\hilh)$ case, since it is known that the stronger condition ``$g\in C^1(\mathbb{R})$ with bounded derivative" is not sufficient for the existence of $\psi'$ in the operator norm, see \cite{Far}. A generalization of the aforementioned result for the higher order differentiability of $\psi$ has been established in \cite{Co22}, where it was shown that if $g$ is $n$-times differentiable with bounded derivatives $g',\ldots, g^{(n)}$, then so does $\psi$. Surprisingly, it appears that the corresponding result for functions of unitaries was not known, even in the case $n=1$ and in the Hilbert-Schmidt case $\mathcal{S}^2(\hilh)$. Namely, if we drop the assumption of continuity of the derivative of $f : \mathbb{T} \to \mathbb{C}$, do we have the differentiability of $\varphi$ in $\mathcal{S}^2(\hilh)$ or even in $\mathcal{S}^p(\hilh)$?

In this paper, we solve this last question in two ways: first by requiring the minimal assumptions on $f$, and secondly by obtaining the $n$th order differentiability for the associated operator function. We prove (see Theorem \ref{main}) that if $1<p<\infty$, $f$ is an $n$-times differentiable function on $\mathbb{T}$ with a bounded $n$th derivative $f^{(n)}$ and $U : \mathbb{R} \to \mathcal{U}(\hilh)$ is such that $\tilde{U} : t\in \mathbb{R} \mapsto U(t)-U(0) \in \mathcal{S}^p(\hilh)$ is $n$-times differentiable, then the operator valued function
$$
\varphi : t\in \mathbb{R} \mapsto f(U(t)) - f(U(0)) \in \mathcal{S}^p(\hilh)
$$
is $n$-times differentiable on $\mathbb{R}$. Moreover, if $U$ has bounded derivatives, then so does $\varphi$. We show that the explicit formula for the derivatives of $\varphi$ given as a sum of multiple operator integrals and that were obtained with stronger assumptions in \cite{CCGP, PoSkSuTo17} also hold true at the degree of generality aimed at in this paper. Note that this result is optimal: it is clear that if $\varphi$ is differentiable for every differentiable function $U$, then $f$ itself must be differentiable. In particular, this paper settles the question of $\mathcal{S}^p$-differentiability for functions of unitaries. Additionally, we explain, in Remark \ref{Taylorform}, how to obtain a representation of the Taylor remainder
$$
R_{n,f,U}(t):=f(U(t))-f(U(0))-\sum_{m=1}^{n-1} \dfrac{1}{m!}\varphi^{(m)}(0),
$$
as well as an estimate of their $\mathcal{S}^p$-norm in the case $U(t)=e^{itA}U$.\\

To achieve our results, we first have to establish important properties of multiple operator integrals, such as their boundedness on Schatten classes when they are associated to divided differences, and some of their properties that will be suitable to study the differentiability of operator functions. Some of the properties are similar to those we can find in \cite{CCGP}, however, in this more general setting, the proofs will require more care. In particular, our approach uses the construction of multiple operator integrals as defined in \cite{CoLeSu21}, which is appropriate to our study as it is very general. Next, we will show that with the help of a Cayley transform, we can use the selfadjoint analogue of our main result, proved in \cite{Co22}, to obtain our result in a particular case. This step is crucial and this is where the biggest differences appear between the case when $f$ only has a bounded $n$th derivative, and the case when $f$ has more regularity such as $f\in C^n(\mathbb{T})$. In the latter case, one can approximate $f$ and its derivatives uniformly (which yields stronger results), while in the case when the assumptions are relaxed, the approach of \cite{Co22, KiPoShSu14} rests on the approximation of the operators appearing in the $\mathcal{S}^p$-perturbation. Finally, the main result, Theorem \ref{main}, will follow from a careful approximation of the path of unitaries.\\

The paper is organized as follows: in Section \ref{Sec2}, we give the definition of the divided differences of a function $f$ and show that they can be approximated by more regular functions in Lemma \ref{approxdivdiff} and Lemma \ref{approxdivdiff2}. In Section \ref{Sec3}, we recall the definition of multiple operator integrals and establish some properties such as their $\mathcal{S}^p$-boundedness in Theorem \ref{MOISpBound} and an important perturbation formula in Proposition \ref{PerturbationFormula}. In Section \ref{sectionselfadj}, we generalize the main result of \cite{Co22} to be able to apply it in Proposition \ref{simplecasemainthm}, which is a weaker version of our main result. Finally, Section \ref{sectionmainresult} is dedicated to the proof of Theorem \ref{main}. The proof will require two auxiliary results, Proposition \ref{Approxunit} and Lemma \ref{intermlemma}, which are the first steps towards an approximation argument used in the proof of our main result.

\bigskip

\noindent \textbf{Notations and conventions.}
\medskip

- Whenever $Z$ is a set and $W\subset Z$ a subset, we let $\chi_W : Z \to \{ 0, 1 \}$ denote the characteristic function of $W$.

\smallskip

- As recalled at the beginning of the Introduction, $\mathcal{S}^p(\hilh)$ will denote the $p$-Schatten class on a complex separable Hilbert space $\hilh$, and $\mathcal{S}^p_{sa}(\hilh)$ will be the subset of selfadjoint operators in $\mathcal{S}^p(\hilh)$. The $\mathcal{S}^p$-norm of an element $K\in \mathcal{S}^p(\hilh)$ is denoted by $\|K\|_p$.

\smallskip

- Similarly, $\mathcal{B}(\hilh)$ is the Banach space of bounded linear operators $A : \hilh \to \hilh$ equipped with the operator norm, denoted by $\|A\|$. We let $\mathcal{B}_{sa}(\hilh)$ denote the subset of bounded selfadjoint operators on $\hilh$.

\smallskip

- Let $f : \mathbb{T} \to \mathbb{C}$ be a function. The derivative of $f$ at $z_{0}\in\cir$ is the limit
\begin{align}\label{ref14}
f'(z_0):=\lim_{z\in \cir, ~ z\to z_{0}}\dfrac{f(z)-f(z_{0})}{z-z_{0}},
\end{align}
provided it exists.

\smallskip

- If $\varphi : \mathbb{R} \to \mathcal{S}^p(\hilh)$ is an $\mathcal{S}^p(\hilh)$-valued function, we will say that $\varphi$ is differentiable at $s\in \mathbb{R}$ if the limit
$$
\varphi'(s) :=  \lim_{t\to s} \dfrac{\varphi(t)-\varphi(s)}{t-s}
$$
exists in $\mathcal{S}^p(\hilh)$. In that case, $\varphi'(s) \in \mathcal{S}^p(\hilh)$.

\smallskip

- If $T \in \mathcal{B}(\hilh)$, we let $\sigma(T)$ denote the spectrum of $T$. In particular, if $T\in \mathcal{U}(\hilh)$, $\sigma(U)\subset \mathbb{T}$. 

\smallskip

- For any $k\in\N$, we will use the notation $(T)^{k}=\underbrace{T,\ldots,T}_{\text{$k$}}$.

\smallskip

- Let $n\in \mathbb{N}$ and let $X_1, \ldots, X_n, Y$ be Banach spaces. We let $\mathcal{B}_{n}(X_1\times\cdots\times X_n,Y)$ denote the space of bounded $n$-linear operators $T : X_1 \times \cdots \times X_n \to Y$, equipped with the norm
$$
\|T\|_{\mathcal{B}_{n}(X_1\times\cdots\times X_n,Y)}:=\sup_{\|x_i\|\le 1,\,1\le i\le n}\|T(x_1,\ldots,x_n)\|.
$$
We will sometimes write $\|T\|$ for the norm of $T$ when no confusion can occur. In the case when $X_1=\cdots=X_n=Y$, we will simply denote this space by $\mathcal{B}_{n}(Y)$. Finally, note that $\mathcal{B}_n(\mathcal{S}^2(\mathcal{H}))$ is a dual space, see \cite[Section 3.1]{CoLeSu21} for details.

\smallskip

\section{Divided differences and approximation}\label{Sec2}

In this section, we first recall the definition of the divided differences of a function $f$ and its properties. Then, we will give the construction of two sequences of elements of $C^n(\mathbb{T})$ which approximate, in a certain sense, the divided differences of a function $f$ with bounded $n$th derivative. Both constructions have advantages and disadvantages, as explained before each statement. \\

Let $f : \mathbb{T} \to \mathbb{C}$ be a function defined on $\mathbb{T}$. We define its divided difference $f^{[n]} : \mathbb{T}^{n+1} \to \mathbb{C}$ recursively as follows. First, by convention $f^{[0]}=f$. Next, if $f$ is differentiable, $f^{[1]}:\mathbb{T}^{2}\to\C$ is defined by
\begin{align*}
f^{[1]}(\lambda_1,\lambda_2):=\begin{cases*}
\frac{f(\lambda_1)-f(\lambda_2)}{\lambda_1-\lambda_2}\ \ \ \text{if}\ \lambda_1\neq\lambda_2\\
f'(\lambda_1)\ \ \ \ \ \ \ \ \text{if}\ \lambda_1=\lambda_2,
\end{cases*}\quad\lambda_1,\lambda_2\in\cir.
\end{align*}
Now, let $n\in \mathbb{N}$ and assume that $f$ is $n$-times differentiable on $\mathbb{T}$. The divided difference of $n$th order $f^{[n]}:\mathbb{T}^{n+1}\to\C$ is defined by
\begin{align*}
f^{[n]}(\lambda_1,\lambda_2,\ldots,\lambda_{n+1}):=\begin{cases*}
\frac{f^{[n-1]}(\lambda_1,\lambda_{3},\ldots,\lambda_{n+1})-f^{[n-1]}(\lambda_2,\lambda_{3},\ldots,\lambda_{n+1})}{\lambda_1-\lambda_2}\ \ \ \text{if}\ \lambda_1\neq\lambda_2\\
\partial_{1}f^{[n-1]}(\lambda_1,\lambda_{3},\ldots,\lambda_{n+1})\ \ \ \ \ \ \ \ \ \ \ \ \ \ \ \text{if}\ \lambda_1=\lambda_2
\end{cases*}
\end{align*}
for all $\lambda_1,\ldots,\lambda_{n+1}\in\cir$, where $\partial_{i}$ stands for the partial derivative with respect to the $i$th variable.

The function $f^{[n]}$ is symmetric in the $n+1$ variables $(\lambda_1,\ldots, \lambda_{n+1})$, it is measurable, and $f^{[n]}$ is bounded if and only if $f^{(n)}$ is bounded. Indeed, it follows from \cite[Theorem 2.1]{Curtiss} that there exists a constant $d_n$ such that
\begin{equation}\label{majdivdiff}
\|f^{[n]}\|_{L^{\infty}(\mathbb{T}^{n+1})} \leq d_n \|f^{(n)}\|_{L^{\infty}(\mathbb{T})}.
\end{equation}
In \cite{Curtiss}, the estimate for $|f^{[n]}(\lambda_1,\lambda_2,\ldots,\lambda_{n+1})|$ was obtained for distinct points $\lambda_i$, but when $f$ is $n$-times differentiable, the same inequality readily extends to every point of $\mathbb{T}^{n+1}$.\\

In the following, we give the first construction of a sequence $(f_j)_j$ which will approximate the derivatives of $f$ and its divided differences. This construction will allow us to obtain a satisfactory bound for the $n$th divided difference of $f_j$, which in turn will allow us to get a certain bound in Theorem \ref{MOISpBound}.

\begin{lma}\label{approxdivdiff}
Let $n\in \mathbb{N}$ and let $f : \mathbb{T} \to \mathbb{C}$ be a $n$-times differentiable function such that $f^{(n)}$ is bounded. Then there exists a sequence $(f_j)_j$ of trigonometric polynomials on $\mathbb{T}$ such that
\begin{enumerate}
\item For every $1\leq k \leq n-1$, the sequence $(f_j^{[k]})_j$, is uniformly convergent to $f^{[k]}$ on $\mathbb{T}^{k+1}$.
\item The sequence $(f_j^{[n]})_j$ is pointwise convergent to $f^{[n]}$ on $\mathbb{T}^{n+1} \setminus \{ (\lambda_1,\ldots, \lambda_{n+1}) \mid \lambda_1 = \cdots = \lambda_{n+1} \}$.
\item For every $j$,
$$\|f_j^{[n]}\|_{L^{\infty}(\mathbb{T}^{n+1})} \leq d_n \|f_j^{(n)}\|_{L^{\infty}(\mathbb{T})} \leq d_n \|f^{(n)}\|_{L^{\infty}(\mathbb{T})},$$
where $d_n$ is the constant given in \eqref{majdivdiff}.
\end{enumerate}
\end{lma}

\begin{proof}
Define, for every $j\in \mathbb{N}$, $f_j := f\ast F_j$ where $F_j$ is the Fej\'er kernel, that is,
$$
\forall z=e^{i\theta} \in \mathbb{T}, \ f_j(z) = \int_0^{2\pi} f(e^{it})F_j(\theta-t) \frac{dt}{2\pi} = \int_0^{2\pi} f(e^{i(\theta-t)})F_j(t) \frac{dt}{2\pi}.
$$
For every $j$, $f_j$ is a trigonometric polynomial. Moreover, since $f^{(n)}$ is bounded, it is a well-known fact that $f_j$ is $n$-times differentiable on $\mathbb{T}$ and for every $1\leq k \leq n$,
$$
\forall z=e^{i\theta} \in \mathbb{T}, \ f_j^{(k)}(z) = \int_0^{2\pi} f^{(k)}(e^{i(\theta-t)}) e^{-ikt}F_j(t) \frac{dt}{2\pi} = (f^{(k)} \ast F_{j,k})(z),
$$
where $F_{j,k}(t) = e^{-ikt}F_j(t)$. In particular, according to \eqref{majdivdiff} and by Young's inequality, we have, for every $1\leq k \leq n$,
\begin{align*}
\|f_j^{[k]}\|_{L^{\infty}(\mathbb{T}^{k+1})} \leq d_k \|f_j^{(k)}\|_{L^{\infty}(\mathbb{T})} \leq d_k \|f^{(k)}\|_{L^{\infty}(\mathbb{T})} \|F_{j,k}\|_{L^{1}(\mathbb{T})}
& = d_k \|f^{(k)}\|_{L^{\infty}(\mathbb{T})} \|F_{j}\|_{L^{1}(\mathbb{T})}\\
& = d_k \|f^{(k)}\|_{L^{\infty}(\mathbb{T})}.
\end{align*}
Next, it is a classical fact that for every $1 \leq k \leq n-1$,
$$
f_j^{(k)} \underset{j\to \infty}{\longrightarrow} f^{(k)} \quad \text{uniformly on} \ \mathbb{T}.
$$
By \eqref{majdivdiff}, this yields
$$
\|f^{[k]} - f_j^{[k]}\|_{L^{\infty}(\mathbb{T}^{k+1})} = \|(f-f_j)^{[k]}\|_{L^{\infty}(\mathbb{T}^{k+1})} \leq d_k \|f^{(k)} - f_j^{(k)}\|_{L^{\infty}(\mathbb{T})} \underset{j\to +\infty}{\longrightarrow} 0.
$$
Now, let $(\lambda_1, \ldots, \lambda_{n+1}) \in \mathbb{T}^{n+1}$ be outside the diagonal of $\mathbb{T}^{n+1}$. Let $1\leq i \leq n$ be such that $\lambda_i \neq \lambda_{i+1}$. It follows from the symmetry of $f_j^{[n]}$ that
\begin{align*}
& f_j^{[n]}(\lambda_1,\ldots, \lambda_{n+1}) = \frac{f_j^{[n-1]}(\lambda_1,\ldots, \lambda_i, \lambda_{i+2},\ldots,\lambda_{n+1})-f_j^{[n-1]}(\lambda_1,\ldots, \lambda_{i-1}, \lambda_{i+1},\ldots,\lambda_{n+1})}{\lambda_i-\lambda_{i+1}}.
\end{align*}
Hence, the pointwise convergence of $(f_j^{[n-1]})_j$ to $f^{[n-1]}$ implies the convergence of the sequence $(f_j^{[n]}(\lambda_1,\lambda_2,\ldots,\lambda_{n+1}))_j$ to $f^{[n]}(\lambda_1,\lambda_2,\ldots,\lambda_{n+1})$.
\end{proof}

The next lemma gives the construction of another sequence of functions $(f_j)_j$ whose advantage is that $(f_j^{[n]})_j$ is pointwise convergent to $f^{[n]}$ everywhere. However, it is not clear that we can estimate the derivatives $f_j^{(n)}$ as efficiently as in Lemma \ref{approxdivdiff}. For that reason, and even if we can have a better estimate, we will only prove that they are bounded, which is enough for our purpose. This result will be useful in Section \ref{sectionselfadj}, as it will allow us to circumvent certain combinatorial and computational difficulties, see Proposition \ref{simplecasemainthm}.

\begin{lma}\label{approxdivdiff2}
Let $n\in \mathbb{N}$ and let $f : \mathbb{T} \to \mathbb{C}$ be a $n$-times differentiable function such that $f^{(n)}$ is bounded. Then there exists a sequence $(f_j)_j \subset C^n(\mathbb{T})$ such that
\begin{enumerate}
\item For every $1\leq k \leq n-1$, the sequence $(f_j^{[k]})_j$, is uniformly convergent to $f^{[k]}$ on $\mathbb{T}^{k+1}$.
\item The sequence $(f_j^{[n]})_j$ is pointwise convergent to $f^{[n]}$ on $\mathbb{T}^{k+1}$.
\item There exist a constant $M>0$ such that, for every $1 \leq k \leq n$ and every $j\in \mathbb{N}$,
$$\|f_j^{(k)}\|_{L^{\infty}(\mathbb{T})} \leq M.$$
\end{enumerate}
\end{lma}

\begin{proof}
For a function $g : \mathbb{T} \to \mathbb{C}$, we let $\tilde{g} : \mathbb{R} \to \mathbb{C}$ to be the $2\pi$-periodic function defined for every $t\in \mathbb{R}$ by $\tilde{g}(t) = g(e^{it})$. Then $g$ is $n$-times differentiable on $\mathbb{T}$ if and only if $\tilde{g}$ is $n$-times differentiable on $\mathbb{R}$. Moreover, by induction (or using Fa\`a di Bruno's formula), we can prove that for every $1\leq k\leq n$, there exist constants $a_{1,k}, \ldots, a_{k,k}, b_{1,k}, \ldots, b_{k,k} \in \mathbb{C}$ (which we do not need to explicit) such that, for every $e^{it}\in \mathbb{T}$,
\begin{equation}\label{fand2pif}
(\tilde{g})^{(k)}(t) = \sum_{p=1}^k a_{p,k} e^{ipt} g^{(p)}(e^{it}) \quad \text{and} \quad g^{(k)}(e^{it}) = e^{-ikt}\sum_{p=1}^k b_{p,k} (\tilde{g})^{(p)}(t).
\end{equation}
Define, for any $j\in \mathbb{N}$, $\tilde{f_j} : \mathbb{R}\to \mathbb{C}$ by
$$
\forall t \in \mathbb{R}, \ \tilde{f_j}(t) = j\int_{0}^t (\tilde{f}(u+1/j)-\tilde{f}(u)) du + f(0).
$$
Then $\tilde{f_j}$ is $2\pi$-periodic, $\tilde{f_j} \in C^n(\mathbb{R})$ and for every $1\leq k \leq n$ and every $t\in \mathbb{R}$,
\begin{equation}\label{convsimplenthderderivatives}
(\tilde{f_j})^{(k)}(t) = j \left((\tilde{f})^{(k-1)}(t+1/j)-(\tilde{f})^{(k-1)}(t) \right).
\end{equation}
It is then easy to check that for every $1\leq k \leq n-1$,
\begin{equation}\label{convsimplenthder2}
(\tilde{f_j})^{(k)} \underset{j\to +\infty}{\longrightarrow} (\tilde{f})^{(k)}
\quad \text{uniformly on} \ \mathbb{R}
\end{equation}
and
\begin{equation}\label{convsimplenthder}
(\tilde{f_j})^{(n)} \underset{j\to +\infty}{\longrightarrow} (\tilde{f})^{(n)}
\quad \text{pointwise on} \ \mathbb{R}.
\end{equation}
Moreover, for every $1\leq k\leq n$,
\begin{equation}\label{boundderivativesapprox}
\forall t\in \mathbb{R}, \ |(\tilde{f_j})^{(k)}(t)| \leq \| (\tilde{f})^{(k)} \|_{L^{\infty}(\mathbb{R})}.
\end{equation}

Let us show that the sequence $(f_j)_j$, where $f_j(e^{it}) = \tilde{f_j}(t)$, satisfies conditions $(1),(2)$ and $(3)$. 
First, $f_j$ is $n$-times differentiable on $\mathbb{T}$ and according to \eqref{fand2pif}, we have, for every $1\leq k\leq n$,
$$
f_j^{(k)}(e^{it}) = e^{-ikt}\sum_{p=1}^k b_{p,k} (\tilde{f_j})^{(p)}(t).
$$
It follows that $f_j^{(k)}$ is continuous on $\mathbb{T}$ so that $f_j \in C^n(\mathbb{T})$. Moreover, for $1\leq k\leq n-1$ and by \eqref{convsimplenthder2}, $(f_j^{(k)})_j$ is uniformly convergent to the function $$z=e^{it} \mapsto e^{-ikt}\sum_{p=1}^k b_{p,k} (\tilde{f})^{(p)}(t) = f^{(k)}(e^{it}),$$
and similarly, by \eqref{convsimplenthder2} and \eqref{convsimplenthder},
\begin{equation}\label{convsimplenthder3}
f_j^{(n)} \underset{j\to +\infty}{\longrightarrow} f^{(n)} \quad \text{pointwise on} \ \mathbb{T}.
\end{equation}
Hence, according to \eqref{majdivdiff}, we have that for $1\leq k \leq n-1$,
$$
\|f^{[k]} - f_j^{[k]}\|_{L^{\infty}(\mathbb{T}^{k+1})} = \|(f-f_j)^{[k]}\|_{L^{\infty}(\mathbb{T}^{k+1})} \leq d_k \|f^{(k)} - f_j^{(k)}\|_{L^{\infty}(\mathbb{T})} \underset{j\to +\infty}{\longrightarrow} 0,
$$
which gives $(1)$.
As in the proof of Lemma \ref{approxdivdiff}, it also follows that $(f_j^{[n]})_j$ is pointwise convergent to $f^{[n]}$ outside the diagonal of $\mathbb{T}^{k+1}$ and if $(\lambda, \ldots, \lambda) \in \mathbb{T}^{n+1}$, we have, by \eqref{convsimplenthder3},
$$
f_j^{[n]}(\lambda, \ldots, \lambda) = \frac{1}{n!} f_j^{(n)}(\lambda) \underset{j\to +\infty}{\longrightarrow} \frac{1}{n!} f^{(n)}(\lambda) = f^{[n]}(\lambda, \ldots, \lambda),
$$
which proves that $(f_j)_j$ satisfies $(2)$.

Finally, by \eqref{boundderivativesapprox}, the sequences $((\tilde{f_j})^{(k)})_j$, $1\leq k \leq n$, are uniformly bounded on $\mathbb{R}$ and by \eqref{fand2pif}, this implies that the sequences $(f_j^{(k)})_j$, $1\leq k \leq n$, are uniformly bounded on $\mathbb{T}$. This yields $(3)$ and finishes the proof of the Lemma.
\end{proof}

\section{Multiple operator integrals}\label{Sec3}

In this section, we first recall the definition of multiple operator integrals mappings as constructed in \cite[Section 3]{CoLeSu21}. The other approaches to operator integration require a certain regularity of the symbol, while this construction is more general and thus align with this paper's scope. We refer to \cite{SkToBook} for other approaches, as well as the references therein. Next, we extend the result on the $\mathcal{S}^p$-boundedness of such mapping when the symbol is a divided difference $f^{[n]}$ for a (non-continuously) $n$-times differentiable function $f$ with bounded $n$th derivative. Finally, we prove an important perturbation formula and give some of its consequences, which are key for our analysis.

\subsection{Definition and background}\label{MOI}

Let $A$ be a normal operator on $\hilh$. In this paper, $A$ will be a unitary operator most of the time, but we will also need the case of selfadjoint operators in Section \ref{sectionselfadj}. Denote by $E^A$ its spectral measure. For any bounded Borel function $f : \sigma(A) \to \mathbb{C}$, one defines an element $f(A) \in \mathcal{B}(\hilh)$ by setting
\begin{align*}
f(A):=\int_{\sigma(A)}f(t)~dE^{A}(t).
\end{align*}
According to \cite[Section 15]{ConwayBook}, there exists a positive finite measure $\lambda_A$ on the Borel subsets of $\sigma(A)$ such that $E^A$ and $\lambda_A$ have the same sets of measure zero. If $f : \sigma(A) \to \mathbb{C}$ is bounded, then by \cite[Theorem 15.10]{ConwayBook}, the operator $f(A)$ only depends on the class of $f$ in $L^{\infty}(\lambda_{A})$ and it induces a $w^{\ast}$-continuous $\ast$-representation
\begin{align*}
f\in L^{\infty}(\lambda_{A})\mapsto f(A)\in\bh.
\end{align*}
The measure $\lambda_A$ is called a scalar-valued spectral measure for $A$.

Let $n\in\N$, and let $A_1,\ldots,A_{n+1}$ be normal operators on $\hilh$ with scalar-valued spectral measures $\lambda_{A_1},\ldots,\lambda_{A_{n+1}}$ Let
\begin{align*}
\Gamma:L^{\infty}(\lambda_{A_1})\otimes\cdots\otimes L^{\infty}(\lambda_{A_{n+1}})\to\mathcal{B}_n(\Sp^2(\hilh))
\end{align*}
be the linear map such that for any $f_{i}\in L^{\infty}(\lambda_{A_i}),~1\le i\le n+1$, and for any $K_1,\ldots,K_n\in\Sp^2(\hilh)$,
\begin{align}
&\left[\Gamma(f_1\otimes\cdots\otimes f_{n+1})\right](K_1,\ldots,K_n)=f_1(A_1)K_1f_2(A_2)K_2\cdots f_n(A_n)K_nf_{n+1}(A_{n+1}).
\end{align}
According to {\normalfont\cite[Proposition 3.4 and Corollary 3.9]{CoLeSu21}}, $\Gamma$ extends to a unique $w^{\ast}$-continuous isometry denoted by
\begin{align*}
\Gamma^{A_1,\ldots,A_{n+1}}:L^{\infty}\left(\prod_{i=1}^{n+1}\lambda_{A_i}\right)\to\mathcal{B}_n(\Sp^2(\hilh)).
\end{align*}
As recalled in the introduction, $\mathcal{B}_n(\Sp^2(\hilh))$ is a dual space, and the $w^{\ast}$-continuity of $\Gamma^{A_1,\ldots,A_{n+1}}$ means that if a net $(\varphi_{i})_{i\in I}$ in $L^{\infty}\left(\prod_{i=1}^{n+1}\lambda_{A_i}\right)$ converges to $\varphi\in L^{\infty}\left(\prod_{i=1}^{n+1}\lambda_{A_i}\right)$ in the $w^{\ast}$-topology, then for any $K_1,\ldots,K_n\in\Sp^{2}(\hilh)$, the net
\begin{align*}
\left(\left[\Gamma^{A_1,\ldots,A_{n+1}}(\varphi_i)\right](K_1,\ldots,K_n) \right)_{i\in I}
\end{align*}
converges to $\left[\Gamma^{A_1,\ldots,A_{n+1}}(\varphi)\right](K_1,\ldots,K_n)$ weakly in $\Sp^{2}(\hilh)$.

\begin{dfn}\label{DefMOI}
For $\varphi\in L^{\infty}\left(\prod_{i=1}^{n+1}\lambda_{A_i}\right)$, the transformation $\Gamma^{A_1,\ldots,A_{n+1}}(\varphi)$ is called multiple operator integral associated to $A_1
,\ldots,A_{n+1}$ and $\varphi$. The element $\varphi$ is sometimes referred to as the symbol of the multiple operator integral.
\end{dfn}

To conclude this subsection, note that one can define
$$
\Gamma^{A_1,\ldots,A_{n+1}}(\varphi):\Sp^{2}(\hilh)\times\cdots\times\Sp^{2}(\hilh)\to\Sp^{2}(\hilh)
$$
for any bounded Borel function $\varphi:\mathcal{U}\to\C$ such that $\prod_{i=1}^{n+1} \sigma(A_i) \subset \mathcal{U}$ by setting
$$
\Gamma^{A_1,\ldots,A_{n+1}}(\varphi) := \Gamma^{A_1,\ldots,A_{n+1}}(\widetilde{\varphi}),
$$
where $\widetilde{\varphi}$ is the class of its restriction $\varphi\vert_{\sigma(A_1)\times\cdots\times\sigma(A_{n+1})}$ in $L^{\infty}\left(\prod_{i=1}^{n+1}\lambda_{A_i}\right)$.

\subsection{$\mathcal{S}^p$-boundedness and perturbation estimate}\label{secpert}

Let $1<p<\infty$, $n\in \mathbb{N}$, and let $U_1,\ldots,U_{n+1}$ be unitaries on $\hilh$. In this subsection, we first establish that for every $n$-times differentiable function $f$ on $\mathbb{T}$ with bounded $n$th derivative, for the symbol $\varphi=f^{[n]}$, we have $\Gamma^{U_1,\ldots,U_{n+1}}(f^{[n]}) \in \mathcal{B}_n(\mathcal{S}^p(\hilh))$.

Precisely, and more generally, we will show the following. If $1<p,p_j<\infty$, $j=1,\ldots,n$ are such that $\frac{1}{p}=\frac{1}{p_{1}}+\cdots+\frac{1}{p_{n}}$ and $\mathcal{S}^2(\mathcal{H}) \cap \mathcal{S}^{p_j}(\mathcal{H})$ is equipped with the $\|.\|_{p_j}$-norm, the $n$-linear mapping
$$
\Gamma^{U_1, \ldots, U_{n+1}}(f^{[n]}) : \left( \mathcal{S}^2(\mathcal{H}) \cap \mathcal{S}^{p_1}(\mathcal{H}) \right) \times \cdots \times \left( \mathcal{S}^2(\mathcal{H}) \cap \mathcal{S}^{p_n}(\mathcal{H}) \right) \rightarrow \mathcal{S}^p(\mathcal{H}),
$$
is bounded. In particular, by density, it uniquely extends to an element
$$\Gamma^{U_1,\ldots,U_{n+1}}(f^{[n]})\in\mathcal{B}_{n}(\Sp^{p_{1}}(\hilh)\times\cdots\times\Sp^{p_{n}}(\hilh),\Sp^p(\hilh)).$$
This result has been established for $n=1$ and a Lipschitz function $f$ on $\mathbb{T}$ in \cite[Theorem 2]{AyCoSu16}, and in \cite[Theorem 2.3]{CCGP} for a general $n\in \mathbb{N}$ and a function $f$ with continuous $n$th derivative $f^{(n)}$. The selfadjoint counterpart of this result, that is for an $n$-times differentable function $g : \mathbb{R} \to \mathbb{C}$ with bounded derivatives $g', \ldots, g^{(n)}$, has been proved in \cite[Theorem 2.7]{Co22}. We will need this fact in Section \ref{sectionselfadj} when considering functions of selfadjoint operators.\\

Let us start with the following Lemma which is the unitary analogue of \cite[Lemma $2.3$]{Co22}. It will be used throughout this paper. Note that it holds true even for normal operators, with the same proof.

\begin{lma}\label{LemmeUB}
Let $n \in \mathbb{N}$ and let $p_1, \ldots, p_n, p \in (1, \infty)$ be such that $\frac{1}{p}=\frac{1}{p_{1}}+\cdots+\frac{1}{p_{n}}$. Let $U_1, \ldots, U_{n+1}$ be unitary operators on $\mathcal{H}$. Let $(\varphi_k)_{k\geq 1}, \varphi \in L^{\infty}\left(\prod_{i=1}^{n+1}\lambda_{U_i}\right)$ and  assume that $(\varphi_k)_k$ is $w^*$-convergent to $\varphi$ and that there exists $C\geq 0$ such that, for every $k\geq 1$,
$$
\| \Gamma^{U_1, \ldots, U_{n+1}}(\varphi_k)\|_{\mathcal{B}_n(\mathcal{S}^{p_1} \times \cdots \times \mathcal{S}^{p_n}, \mathcal{S}^{p})} \leq C.
$$
Then $\Gamma^{U_1, \ldots, U_{n+1}}(\varphi) \in \mathcal{B}_n(\mathcal{S}^{p_1} \times \cdots \times \mathcal{S}^{p_n}, \mathcal{S}^p)$ and
\begin{align*}
\| \Gamma^{U_1, \ldots, U_{n+1}}(\varphi)\|_{\mathcal{B}_n(\mathcal{S}^{p_1} \times \cdots \times \mathcal{S}^{p_n}, \mathcal{S}^p)} \leq C.
\end{align*}
Moreover, for any $X_i \in \mathcal{S}^{p_i}(\mathcal{H}), 1 \leq i \leq n$,
$$
\left[ \Gamma^{U_1, \ldots, U_{n+1}}(\varphi_k) \right](X_1, \ldots, X_n) \underset{k \to \infty}{\longrightarrow} \left[ \Gamma^{U_1, \ldots, U_{n+1}}(\varphi) \right](X_1, \ldots, X_n)
$$
weakly in $\mathcal{S}^{p}(\mathcal{H})$.
\end{lma}

The following states that \cite[Theorem 2.3]{CCGP} remains true when we drop the assumption of continuity of $f^{(n)}$. It is crucial because it ensures that all the operators that will appear in the rest of the paper belong to $\mathcal{S}^p(\hilh)$.

\begin{thm}\label{MOISpBound}
Let $n\in\N$ and let $f : \mathbb{T} \to \mathbb{C}$ be $n$-times differentiable such that $f^{(n)}$ is bounded. Let $1<p,p_{j}<\infty$, $j=1,\ldots,n$ be such that $\frac{1}{p}=\frac{1}{p_{1}}+\cdots+\frac{1}{p_{n}}$. Let $U_1, \ldots, U_{n+1}$ be unitary operators on $\mathcal{H}$. Then
$$\Gamma^{U_1,\ldots,U_{n+1}}(f^{[n]})\in\mathcal{B}_{n}(\Sp^{p_{1}}(\hilh)\times\cdots\times\Sp^{p_{n}}(\hilh),\Sp^p(\hilh))$$
and there exists $C_{p,n}>0$ depending only on $p$ and $n$ such that
\begin{align}\label{MOIBound1}
\left\| \Gamma^{U_1,\ldots,U_{n+1}}(f^{[n]}) \right\|_{\mathcal{B}_{n}(\Sp^{p_{1}}(\hilh)\times\cdots\times\Sp^{p_{n}}(\hilh),\Sp^p(\hilh))} \leq C_{p,n}\left\|f^{(n)}\right\|_{L^{\infty}(\mathbb{T})}.
\end{align}
In particular, $\Gamma^{U_1,\ldots,U_{n+1}}(f^{[n]})\in\mathcal{B}_{n}(\Sp^{p}(\hilh))$, with
\begin{align}\label{MOIBound2}
\left\|\Gamma^{U_1,\ldots,U_{n+1}}(f^{[n]})\right\|_{\mathcal{B}_{n}(\Sp^{p})}\leq C_{p,n} \left\|f^{(n)}\right\|_{L^{\infty}(\mathbb{T})}.
\end{align}
\end{thm}

Before proving this Theorem, we will need the following Lemma. It is certainly well-known to specialists but we include a proof for the convenience of the reader.

\begin{lma}\label{boundeddiag}
Let $n\in\N$ and let $f : \mathbb{T} \to \mathbb{C}$ be $n$-times differentiable such that $f^{(n)}$ is bounded. Let $1<p,p_{j}<\infty$, $j=1,\ldots,n$ be such that $\frac{1}{p}=\frac{1}{p_{1}}+\cdots+\frac{1}{p_{n}}$. Let $U_1, \ldots, U_{n+1}$ be unitary operators on $\mathcal{H}$. Let $\Delta:=\{ (\lambda_1,\ldots, \lambda_{n+1}) \mid \lambda_1=\cdots=\lambda_{n+1} \}$ be the diagonal of $\mathbb{T}^{n+1}$. Then $\Gamma^{U_1,\ldots,U_{n+1}}(f^{[n]} \chi_{\Delta}) \in \mathcal{B}_{n}(\Sp^{p_{1}}(\hilh)\times\cdots\times\Sp^{p_{n}}(\hilh),\Sp^p(\hilh))$ and
$$\left\| \Gamma^{U_1,\ldots,U_{n+1}}(f^{[n]} \chi_{\Delta}) \right\| \leq \frac{1}{n!} \|f^{(n)}\|_{L^{\infty}(\mathbb{T})}.$$
\end{lma}

\begin{proof}
Let $(g_k)_k$ be a sequence of continuous functions converging pointwise to $f^{(n)}$ on $\mathbb{T}^{n+1}$ and such that for every $k$, $\|g_k \|_{L^{\infty}(\mathbb{T})} \leq \| f^{(n)} \|_{L^{\infty}(\mathbb{T})}$ (take e.g. $g_k(z) =  \frac{k(f^{(n-1)}(ze^{i/k})-f^{(n-1)}(z))}{iz}$). Let $\tilde{g_k}$ be defined, for any $(\lambda_1,\ldots, \lambda_{n+1}) \in \mathbb{T}^{n+1}$, by
$$\tilde{g_k}(\lambda_1,\ldots, \lambda_{n+1}) = \frac{1}{n!} g_k(\lambda_1)\chi_{\Delta}(\lambda_1,\ldots, \lambda_{n+1}).$$
Note that
$$
(f^{[n]}\chi_{\Delta})(\lambda_1,\ldots,\lambda_{n+1})=\frac{1}{n!}f^{(n)}(\lambda_1) \chi_{\Delta}(\lambda_1,\ldots,\lambda_{n+1}).
$$
Hence, by the Lebesgue's dominated convergence theorem, $(\tilde{g_k})_k$ $w^*$-converges to $f^{[n]}\chi_{\Delta}$ for the $w^*$-topology of $L^{\infty}(\lambda_{U_1} \times \cdots \times \lambda_{U_{n+1}})$.  In particular, to prove the Lemma, it suffices, according to Lemma \ref{LemmeUB}, to prove that $\Gamma^{U_1,\ldots,U_{n+1}}(\tilde{g_k}) \in \mathcal{B}_{n}(\Sp^{p_{1}}(\hilh)\times\cdots\times\Sp^{p_{n}}(\hilh),\Sp^p(\hilh))$ with norm less than $\frac{1}{n!}\|f^{(n)}\|_{L^{\infty}(\mathbb{T})}$.
To simplify the notations, set $h:=\frac{1}{n!}g_r$ and $\tilde{h}:=\tilde{g_r}$ for some fixed $r\in \mathbb{N}$.

Let $m\in \mathbb{N}$. Let $A_{m,k}:=\{ e^{2i\pi t} \mid \frac{k}{2^m} \leq t < \frac{k+1}{2^m} \}$ and define $P^j_{m,k} := E^{U_j}(A_{m,k})$. Then, for every $1\leq j \leq n+1$, $\displaystyle \sum_{k=0}^{2^m-1} P^j_{m,k} = I_{\hilh}$. Let $2\leq q \leq n$ and let $K \in \mathcal{S}^{p_q}(\hilh)$. Define $\displaystyle U_m(t) =\sum_{k=0}^{2^m-1} e^{ikt} P^q_{m,k}$ and $\displaystyle V_m(t) =\sum_{k=0}^{2^m-1} e^{-ikt} P^{q+1}_{m,k}$. For every $t\in \mathbb{R}$, $U_m(t)$ and $V_m(t)$ are unitaries on $\mathcal{H}$ and we have
$$
U_m(t)KV_m(t) =  \sum_{k,l=0}^{2^m-1} e^{i(k-l)t} P^q_{m,k} K P^{q+1}_{m,l},
$$
so that
$$
\frac{1}{2\pi} \int_0^{2\pi} U_m(t)KV_m(t) dt = \sum_{k=0}^{2^m-1}  P^q_{m,k} K P^{q+1}_{m,k},
$$
which in turns yields
\begin{equation}\label{estimate1}
\left\| \sum_{k=0}^{2^m-1} P^q_{m,k} K P^{q+1}_{m,k} \right\|_{p_q} \leq \frac{1}{2\pi}  \int_0^{2\pi} \|U_m(t)KV_m(t) \|_{p_q} dt = \| K \|_{p_q}.
\end{equation}
In the case $q=1$, one defines $\displaystyle \widetilde{U}_m(t):= \sum_{k=0}^{2^m-1} h\left( e^{\frac{ik}{2^m}} \right) e^{ikt} P^1_{m,k}$ and $\displaystyle V_m(t) =\sum_{k=0}^{2^m-1} e^{-ikt} P^2_{m,k}$. Then $\| \widetilde{U}_m(t) \| \leq \|h\|_{L^{\infty}(\mathbb{T})}$ and proceeding as above, we get the estimate
\begin{equation}\label{estimate2}
\left\| \sum_{k=0}^{2^m-1} h\left( e^{\frac{ik}{2^m}} \right) P^1_{m,k} K P^2_{m,k} \right\|_{p_1} \leq \|h\|_{L^{\infty}(\mathbb{T})} \| K \|_{p_1} \leq \frac{1}{n!}\|f^{(n)}\|_{L^{\infty}(\mathbb{T})} \|K\|_{p_1}.
\end{equation}
Next, let $\displaystyle E_{m,k} = \prod_{i=1}^{n+1} A_{m,k}$ be the Cartesian product of $n+1$ copies of $A_{m,k}$. Define
$$\varphi_m := \sum_{k=0}^{2^m-1} h\left( e^{\frac{ik}{2^m}} \right) \chi_{E_{m,k}}.$$
Let, for every $1\leq i \leq n$, $K_i \in \mathcal{S}^{p_i}(\hilh)$. We have, by definition of multiple operator integrals and by orthogonality,
\begin{align*}
& \left[\Gamma^{U_1,\ldots, U_{n+1}}(\varphi_m)\right](K_1, \ldots, K_n)
 = \sum_{k=0}^{2^m-1} h\left( e^{\frac{ik}{2^m}} \right) P^1_{m,k} K_1 P^2_{m,k} \cdots  P^n_{m,k} K_n P^{n+1}_{m,k} \\
&\hspace{0.5cm} = \left( \sum_{k=0}^{2^m-1} h\left( e^{\frac{ik}{2^m}} \right) P^1_{m,k} K_1 P^2_{m,k} \right) \left( \sum_{k=0}^{2^m-1}  P^2_{m,k} K_2 P^3_{m,k} \right) \cdots \left( \sum_{k=0}^{2^m-1}  P^n_{m,k} K_n P^{n+1}_{m,k} \right).
\end{align*}
It follows from \eqref{estimate1} and \eqref{estimate2} that
$$
\left\| \left[\Gamma^{U_1,\ldots, U_{n+1}}(\varphi_m)\right](K_1, \ldots, K_n) \right\|_p \leq \frac{1}{n!}\|f^{(n)}\|_{L^{\infty}(\mathbb{T})} \|K_1\|_{p_1} \cdots \|K_n\|_{p_n}.
$$
To conclude the proof, notice that $\| \varphi \|_{L^{\infty}(\mathbb{T}^{n+1})} \leq \| \tilde{h} \|_{L^{\infty}(\mathbb{T}^{n+1})}$, and by continuity of $h$, $\varphi_m \underset{m\to +\infty}{\longrightarrow} \tilde{h}=\tilde{g_r}$ pointwise on $\mathbb{T}^{n+1}$. Hence, by the Lebesgue's dominated convergence theorem, $(\varphi_m)_m$ $w^*$-converges to $h$ for the $w^*$-topology of $L^{\infty}(\lambda_{U_1} \times \cdots \times \lambda_{U_{n+1}})$. Using Lemma \ref{LemmeUB}, we get that
$$\left\| \Gamma^{U_1,\ldots,U_{n+1}}(\tilde{g_r}) \right\| \leq \frac{1}{n!} \|f^{(n)}\|_{L^{\infty}(\mathbb{T})},$$
and the conclusion of the Lemma follows.
\end{proof}

We now turn to the proof of Theorem \ref{MOISpBound}.

\begin{proof}[Proof of Theorem \ref{MOISpBound}] Let $\Delta:=\{ (\lambda_1,\ldots, \lambda_{n+1} \mid \lambda_1=\cdots=\lambda_{n+1} \}$. By Lemma \ref{boundeddiag},
$$
\Gamma^{U_1,\ldots,U_{n+1}}(f^{[n]}\chi_{\Delta})\in \mathcal{B}_{n}(\Sp^{p_{1}}(\hilh)\times\cdots\times\Sp^{p_{n}}(\hilh),\Sp^p(\hilh))
$$
with a norm less than or equal to $\frac{1}{n!} \|f^{(n)}\|_{L^{\infty}(\mathbb{T})}$. Hence, it suffices to show the boundedness of $\Gamma^{U_1,\ldots,U_{n+1}}(f^{[n]}(1-\chi_{\Delta}))$.
Let $(f_j)_j$ be the sequence of trigonometric polynomials given by Lemma \ref{approxdivdiff}. Let $T_j := \Gamma^{U_1,\ldots,U_{n+1}}(f_j^{[n]})$ and $\tilde{T_j}:= \Gamma^{U_1,\ldots,U_{n+1}}(f_j^{[n]}\chi_{\Delta}).$ According to \cite[Theorem 2.3]{CCGP} and Lemma \ref{boundeddiag},
$$
T_j, \tilde{T_j} \in \mathcal{B}_{n}(\Sp^{p_{1}}(\hilh)\times\cdots\times\Sp^{p_{n}}(\hilh),\Sp^p(\hilh)), 
$$
and there exists a constant $c_{p,n}>0$ such that 
$$
\|T_j \| \leq c_{p,n}\left\|f_j^{(n)}\right\|_{L^{\infty}(\mathbb{T})} \leq c_{p,n}\left\|f^{(n)}\right\|_{L^{\infty}(\mathbb{T})}
$$
and
$$
\|\tilde{T_j} \| \leq \frac{1}{n!} \left\|f_j^{(n)}\right\|_{L^{\infty}(\mathbb{T})} \leq \frac{1}{n!} \left\|f^{(n)}\right\|_{L^{\infty}(\mathbb{T})}.
$$
Notice that $1-\chi_{\Delta}=\chi_{\Omega}$ where $\Omega := \mathbb{T}^{n+1}\setminus \Delta$, so that, according to Lemma \ref{approxdivdiff}, $f_j^{[n]}(1-\chi_{\Delta})$ is pointwise convergent to $f^{[n]}(1-\chi_{\Delta})$ and
$$
\| f_j^{[n]}(1-\chi_{\Delta}) \|_{L^{\infty}(\mathbb{T}^{n+1})} \leq  \| f_j^{[n]}\|_{L^{\infty}(\mathbb{T}^{n+1})}  \leq d_n \|f^{(n)}\|_{L^{\infty}(\mathbb{T})}.
$$
Hence, $(f_j^{[n]}(1-\chi_{\Delta}))_j$ $w^*$-converges to $f^{[n]}(1-\chi_{\Delta})$ for the $w^*$-topology of $L^{\infty}(\lambda_{U_1} \times \cdots \times \lambda_{U_{n+1}})$. Since $$
\left\| \Gamma^{U_1,\ldots,U_{n+1}}(f_j^{[n]}(1-\chi_{\Delta})) \right\|  = \left\| T_j - \tilde{T_j}\right\| \leq \left( c_{p,n} + \frac{1}{n!} \right) \left\|f^{(n)}\right\|_{L^{\infty}(\mathbb{T})},
$$
it follows from Lemma \ref{LemmeUB} that $\Gamma^{U_1,\ldots,U_{n+1}}(f^{[n]}(1-\chi_{\Delta})) \in \mathcal{B}_{n}(\Sp^{p_{1}}(\hilh)\times\cdots\times\Sp^{p_{n}}(\hilh),\Sp^p(\hilh))$ and
$$
\left\| \Gamma^{U_1,\ldots,U_{n+1}}(f^{[n]}(1-\chi_{\Delta})) \right\| \leq \left( c_{p,n} + \frac{1}{n!} \right) \left\|f^{(n)}\right\|_{L^{\infty}(\mathbb{T})}.
$$
This concludes the proof of the Theorem.
\end{proof}

The next proposition is a crucial perturbation formula. It is key whenever the differentiability of operator functions with $S^p$-perturbation is studied. In the unitary settting, it generalizes \cite[Proposition 3.5]{CCGP} where the result was proved when the function $f$ is an element of $C^n(\mathbb{T})$.

\begin{ppsn}\label{PerturbationFormula}
Let $1<p<\infty$ and $n\ge 2$ be an integer. Let $U_1,\ldots,U_{n-1},U,V \in \mathcal{U}(\hilh)$ be such that $U-V\in\Sp^p(\hilh)$. Let $f : \mathbb{T} \to \mathbb{C}$ be $n$-times differentiable on $\mathbb{T}$ such that $f^{(n)}$ is bounded. Then, for all $K_1,\ldots,K_{n-1}\in\Sp^p(\hilh)$ and for any $1\le i\le n$ we have
\begin{align*}
&\left[\Gamma^{U_1,\ldots,U_{i-1},U,U_i,\ldots,U_{n-1}}(f^{[n-1]})-\Gamma^{U_1,\ldots,U_{i-1},V,U_i,\ldots,U_{n-1}}(f^{[n-1]})\right](K_1,\ldots,K_{n-1})\\
&\hspace*{0.5cm}=\left[\Gamma^{U_1,\ldots,U_{i-1},U,V,U_i,\ldots,U_{n-1}}(f^{[n]})\right](K_1,\ldots,K_{i-1},U-V,K_i,\ldots,K_{n-1}).
\end{align*}
\end{ppsn}

\begin{proof} Fix $1\le i\le n$ and let $(f_j)_j$ be the sequence of elements of $C^n(\mathbb{T})$ given by Lemma \ref{approxdivdiff2}. By \cite[Proposition 2.5]{CCGP}, we have
\begin{align}
\nonumber &\left[\Gamma^{U_1,\ldots,U_{i-1},U,U_i,\ldots,U_{n-1}}(f_j^{[n-1]})-\Gamma^{U_1,\ldots,U_{i-1},V,U_i,\ldots,U_{n-1}}(f_j^{[n-1]})\right](K_1,\ldots,K_{n-1})\\
\label{perturbformulagen}&\hspace*{0.5cm}=\left[\Gamma^{U_1,\ldots,U_{i-1},U,V,U_i,\ldots,U_{n-1}}(f_j^{[n]})\right](K_1,\ldots,K_{i-1},U-V,K_i,\ldots,K_{n-1}).
\end{align}
The sequence $(f_j^{(n-1)})_j$ is uniformly convergent to $f^{(n-1)}$ on $\mathbb{T}$. It follows from Theorem \ref{MOISpBound} that
$$
\Gamma^{U_1,\ldots,U_{i-1},U,U_i,\ldots,U_{n-1}}(f_j^{[n-1]}) \underset{j\to +\infty}{\longrightarrow} \Gamma^{U_1,\ldots,U_{i-1},U,U_i,\ldots,U_{n-1}}(f^{[n-1]})
$$
and
$$
\Gamma^{U_1,\ldots,U_{i-1},V,U_i,\ldots,U_{n-1}}(f_j^{[n-1]}) \underset{j\to +\infty}{\longrightarrow} \Gamma^{U_1,\ldots,U_{i-1},V,U_i,\ldots,U_{n-1}}(f^{[n-1]})
$$
in $\mathcal{B}_{n}(\Sp^{p}(\hilh))$. In particular, they converge pointwise, that is
\begin{align*}
&\left[\Gamma^{U_1,\ldots,U_{i-1},U,U_i,\ldots,U_{n-1}}(f_j^{[n-1]})-\Gamma^{U_1,\ldots,U_{i-1},V,U_i,\ldots,U_{n-1}}(f_j^{[n-1]})\right](K_1,\ldots,K_{n-1})\\
&\hspace*{0.5cm} \underset{j\to +\infty}{\longrightarrow} \left[\Gamma^{U_1,\ldots,U_{i-1},U,U_i,\ldots,U_{n-1}}(f^{[n-1]})-\Gamma^{U_1,\ldots,U_{i-1},V,U_i,\ldots,U_{n-1}}(f^{[n-1]})\right](K_1,\ldots,K_{n-1})
\end{align*}
in $\mathcal{S}^p(\hilh)$. Next, by Lemma \ref{approxdivdiff2} and Lebesgue's dominated convergence theorem, the sequence $(f_j^{[n]})_j$ $w^*$-converges to $f^{[n]}$ for the $w^*$-topology of $L^{\infty}(\lambda_{U_1} \times \cdots \lambda_{U_{i-1}} \times \lambda_{U} \times \lambda_{V} \times \lambda_{U_{U_i}} \times \lambda_{U_{n+1}})$. It follows from Lemma \ref{LemmeUB} that 
$$
\left[\Gamma^{U_1,\ldots,U_{i-1},U,V,U_i,\ldots,U_{n-1}}(f_j^{[n]})\right](K_1,\ldots,K_{i-1},U-V,K_i,\ldots,K_{n-1})
$$
converges weakly (in $\mathcal{S}^p(\hilh)$) to
$$
\left[\Gamma^{U_1,\ldots,U_{i-1},U,V,U_i,\ldots,U_{n-1}}(f^{[n]})\right](K_1,\ldots,K_{i-1},U-V,K_i,\ldots,K_{n-1}).
$$
Hence, taking the limit as $j\to +\infty$ in \eqref{perturbformulagen} in the weak topology of $\mathcal{S}^p(\hilh)$ yields the desired identity. 
\end{proof}

\begin{crl}\label{PerturbationFormula2}
Let $1<p<\infty$ and let $n\ge 2$ be an integer. Let $U,V \in \mathcal{U}(\hilh)$ be such that $U-V\in\Sp^p(\hilh)$. Let $f : \mathbb{T} \to \mathbb{C}$ be $n$-times differentiable on $\mathbb{T}$ such that $f^{(n)}$ is bounded. Then, for all $K_1,\ldots,K_{n-1}\in\Sp^p(\hilh)$,
\begin{align*}
&\left[\Gamma^{(U)^n}(f^{[n-1]})-\Gamma^{(V)^n}(f^{[n-1]})\right](K_1,\ldots,K_{n-1})\\
&\hspace*{0.5cm}= \sum_{i=1}^n \left[\Gamma^{(U)^{i},(V)^{n-i+1}}(f^{[n]})\right](K_1,\ldots,K_{i-1},U-V,K_i,\ldots,K_{n-1}).
\end{align*}
\end{crl}

\begin{proof}
It suffices to write
\begin{align*}
&\left[\Gamma^{(U)^n}(f^{[n-1]})-\Gamma^{(V)^n}(f^{[n-1]})\right](K_1,\ldots,K_{n-1})\\
&\hspace*{0.5cm}= \sum_{i=1}^n \left[\Gamma^{(U)^i, (V)^{n-i}}(f^{[n-1]})-\Gamma^{(U)^{i-1}, (V)^{n-i+1}}(f^{[n-1]})\right](K_1,\ldots,K_{n-1})
\end{align*}
and then apply Proposition \ref{PerturbationFormula}.
\end{proof}

\begin{crl}\label{PerturbationFormula1}
Let $1<p<\infty$ and let $n\ge 2$ be an integer. Let $U_1,\ldots,U_n, V_1, \ldots, V_n \in \mathcal{U}(\hilh)$ be such that $U_i-V_i \in \mathcal{S}^p(\hilh)$ for every $1\leq i \leq n$. Let $f : \mathbb{T} \to \mathbb{C}$ be $n$-times differentiable on $\mathbb{T}$ such that $f^{(n)}$ is bounded. Then, there exists $D_{p,n}>0$ such that
\begin{align*}
\left\| \left[\Gamma^{U_1,\ldots,U_{n}}(f^{[n-1]})-\Gamma^{V_1,\ldots,V_{n}}(f^{[n-1]})\right] \right\|_{\mathcal{B}_{n-1}(\mathcal{S}^p(\hilh))} \leq D_{p,n} \|f^{(n)} \|_{L^{\infty}(\mathbb{T})} \max_{1\leq k \leq n} \|U_k-V_k\|_p.
\end{align*}
\end{crl}

\begin{proof} Let $K_1, \ldots, K_{n-1} \in \mathcal{S}^p(\hilh)$. By Proposition \ref{PerturbationFormula}, we have
\begin{align*}
&\left[\Gamma^{U_1,\ldots,U_{n}}(f^{[n-1]})-\Gamma^{V_1,\ldots,V_{n}}(f^{[n-1]})\right](K_1, \ldots, K_{n-1})\\
&\hspace*{0.5cm} = \sum_{k=1}^{n} \left[\Gamma^{U_1,\ldots,U_k, V_{k+1}, \ldots, V_n}(f^{[n-1]})-\Gamma^{U_1,\ldots,U_{k-1}, V_{k}, \ldots, V_n}(f^{[n-1]}) \right](K_1, \ldots, K_{n-1}) \\
& \hspace*{0.5cm} = \sum_{k=1}^{n} \left[\Gamma^{U_1,\ldots,U_{n-k},V_{n-k}, \ldots, V_n}(f^{[n]})\right](K_1,\ldots,K_{k-1},U_{k}-V_{k},K_{k},\ldots,K_{n-1}).
\end{align*}
By Theorem \ref{MOISpBound}, there exists a constant $C_{p,n}$ such that
\begin{align*}
& \left\| \left[\Gamma^{U_1,\ldots,U_{n}}(f^{[n-1]})-\Gamma^{V_1,\ldots,V_{n}}(f^{[n-1]})\right](K_1,\ldots,K_{n-1}) \right\|_p\\
& \hspace*{0,5cm} \leq n \max_{1\leq k \leq n}  \left\| \left[\Gamma^{U_1,\ldots,U_{n-k},V_{n-k}, \ldots, V_n}(f^{[n]})\right](K_1,\ldots,K_{k-1},U_{k}-V_{k},K_{k},\ldots,K_{n-1}) \right\|_p \\
& \hspace{0,5cm} \leq n C_{p,n} \|f^{(n)} \|_{L^{\infty}(\mathbb{T})}  \max_{1\leq k \leq n} \|U_k-V_k\|_p \|K_1\|_p \cdots \| K_{n-1} \|_p.
\end{align*}
This concludes the proof of the Corollary.
\end{proof}

\begin{rmrk}
Proposition \ref{PerturbationFormula} and Corollary \ref{PerturbationFormula1} also hold true in the case $n=1$. For the perturbation formula, it means that if $U,V \in \mathcal{U}(\hilh)$ are such that $U-V\in \mathcal{S}^p(\hilh)$ and $f : \mathbb{T} \to \mathbb{C}$ is differentiable with bounded $f'$, then
$$
f(U)-f(V) = \left[\Gamma^{U,V}(f^{[1]})\right](U-V).
$$
We refer e.g. to \cite{BiSo3}. Alternatively, for a more recent reference, we can first use \cite[Proposition 2.5]{CCGP} in the case $f\in C^1(\mathbb{T})$ (the proof works verbatim for $n=1$) and then some minor modifications to the proof of Proposition \ref{PerturbationFormula} will give the desired formula. The bound given in Corollary \ref{PerturbationFormula1} in the case $n=1$ simply corresponds to \cite[Theorem 2]{AyCoSu16}.
\end{rmrk}

\section{From selfadjoint to unitary operators}\label{sectionselfadj}

In this section, we will prove a general result on the differentiability of operator functions in the selfadjoint case and then deduce its unitary counterpart using a Cayley transform. This will be the first step towards our main theorem in Section \ref{sectionmainresult}.

Let $\mathcal{A} : \mathbb{R} \to \mathcal{B}_{sa}(\mathcal{H})$ be such that $t\in \mathbb{R} \mapsto \mathcal{A}(t) - \mathcal{A}(0) \in \mathcal{S}^p_{sa}(\mathcal{H})$ is $n$-times differentiable, and let $f$ be an $n$-times differentiable function on $\mathbb{R}$. We will show that the function
$$
\varphi : t\in \mathbb{R} \mapsto f(\mathcal{A}(t)) - f(\mathcal{A}(0)) \in \mathcal{S}^p(\hilh)
$$
is $n$-times differentiable as well. The particular case $\mathcal{A}(t) = A+tK$, where $A\in \mathcal{B}_{sa}(\hilh)$ and $K\in \mathcal{S}^p_{sa}(\hilh)$ is the main result of \cite{Co22}. We will outline the minor changes to make in the proof of \cite[Theorem 3.1]{Co22} as well as in the results therein to obtain our general result in Corollary \ref{FormulaSA3}.

Let us start with the following which is the key step prior to a combinatorial reasoning.

\begin{thm}\label{FormulaSA2}
Let $1 < p < \infty$, let $\mathcal{A} : \mathbb{R} \to \mathcal{B}_{sa}(\mathcal{H})$ be such that $\tilde{\mathcal{A}} : t\in \mathbb{R} \mapsto \mathcal{A}(t) - \mathcal{A}(0) \in \mathcal{S}^p_{sa}(\mathcal{H})$ is differentiable at $0$. Let $n \in\N, n\geq 2$. Let, for every $1\leq i \leq n-1$, $S_i : \mathbb{R} \to \mathcal{S}^p_{sa}(\hilh)$ be differentiable on $\mathbb{R}$. Let $f$ be $n$-times differentiable on $\mathbb{R}$ such that $f^{(n)}$ is bounded and consider the function
$$\varphi : t\in \mathbb{R} \mapsto \left[ \Gamma^{(\mathcal{A}(t))^n}(f^{[n-1]})\right](S_1(t), \ldots, S_{n-1}(t)) \in \mathcal{S}^p(\mathcal{H}).$$
Then $\varphi$ is differentiable on $\mathbb{R}$ and for every $t\in \mathbb{R}$,
\begin{align*}
\varphi'(t) = 
& \sum_{k=1}^{n-1} \left[\Gamma^{(\mathcal{A}(t))^{n}}(f^{[n-1]})\right](S_1(t),\ldots, S_{k-1}(t), S_k'(t), S_{k+1}(t), \ldots, S_{n-1}(t)) \\
& + \sum_{k=1}^{n} \left[\Gamma^{(\mathcal{A}(t))^{n+1}}(f^{[n]})\right](S_1(t),\ldots, S_{k-1}(t), \tilde{\mathcal{A}}'(t), S_{k}(t), \ldots, S_{n-1}(t)).
\end{align*}
\end{thm}

\begin{proof} Let us explain the simple modifications to make in \cite{Co22} to obtain the result. Throughout the proof, we denote $A:=\mathcal{A}(0) \in \mathcal{B}_{sa}(\mathcal{H})$. Define, for $K, X_1,\ldots, X_{n-1} \in \mathcal{S}^p_{sa}$,
$$
\varphi_{K,X_1, \ldots, X_{n-1}} : t\in \mathbb{R} \mapsto \left[ \Gamma^{(A+tK)^n}(f^{[n-1]}) \right] (X_1, \ldots, X_{n-1}).
$$
We first want to prove that $\varphi_{K,X_1, \ldots, X_{n-1}}$ is differentiable at $0$. Assume that for every $K_0$ in a dense subset of $\mathcal{S}^p_{sa}$, $\varphi_{K_0,X_1, \ldots, X_{n-1}} $ is differentiable at $0$ with
$$
\varphi_{K_0,X_1, \ldots, X_{n-1}}'(0) = \sum_{k=1}^n \left[ \Gamma^{(A)^{n+1}}(f^{[n]}) \right] (X_1, \ldots, X_{k-1}, K_0, X_k, \ldots, X_{n-1}).
$$
Then, arguing as in the proof of \cite[Lemma 3.7]{Co22}, one shows that for every  $K\in \mathcal{S}^p_{sa}$, $\varphi_{K,X_1, \ldots, X_{n-1}}$ is differentiable at $t=0$ with the same formula for its derivative. Next, as explained in \cite{KiPoShSu14} or in the proof of \cite[Theorem 3.1]{Co22}, one can choose
$$
\mathcal{F} := \left\lbrace i[A,Y] + Z \mid Y,Z \in \mathcal{S}^p_{sa}(\mathcal{H})) \ \text{and} \ Z \ \text{commutes with} \ A \right\rbrace
$$
as a dense subset of $\mathcal{S}^p_{sa}(\mathcal{H})$. By the latter, we can assume that $K=i[A,Y] + Z\in \mathcal{F}$ and we have to show that $\varphi_{K,X_1, \ldots, X_{n-1}}$ is differentiable at $t=0$. The first part of the proof of \cite[Theorem 3.1]{Co22} applies with the same computations and it tells us that $\varphi_{K,X_1, \ldots, X_{n-1}}$ has a derivative in $0$ if and only if
$$
\xi : t\in \mathbb{R} \mapsto \displaystyle \sum_{k=1}^n \left[ \Gamma^{(A+tZ)^{n-k+1}, (A)^k}(f^{[n]})\right] (X_1, \ldots, X_{n-k}, K, X_{n-k+1}, \ldots, X_{n-1})
$$
has a limit in $0$ (in $\mathcal{S}^p$), and in that case, this limit is $\varphi_{K,X_1, \ldots, X_{n-1}}'(0)$. But notice that by continuity of multiple operator integrals (the operators $\Gamma^{(A+tZ)^{n-k+1}, (A)^k}(f^{[n]})$ are uniformly bounded with respect to $t\in \mathbb{R}$), it is enough to show that this limit exists when $X_1, \ldots, X_{n-1}$ are elements of the dense subset $\mathcal{F}$. Hence, one can write $X_i = i[A,Y_i] + Z_i$. The rest of the proof is similar, with obvious modifications, and shows that $\xi$ has indeed a limit in $0$ equal to
\begin{equation}\label{deriveeautoadj}
\varphi_{K,X_1, \ldots, X_{n-1}}'(0) = \sum_{k=1}^n \left[ \Gamma^{(A)^{n+1}}(f^{[n]})\right] (X_1, \ldots, X_{n-k}, K, X_{n-k+1}, \ldots, X_{n-1}),
\end{equation}
as expected.

Now, let us come back to the function $\varphi$. It is sufficient to prove the formula for $t=0$. Let $K:=\mathcal{A}'(0)$. By a straightforward modification of \cite[Lemma 3.8]{Co22}, $\varphi$ is differentiable at $0$ if and only if
$$\tilde{\varphi} : t\in \mathbb{R} \mapsto \left[ \Gamma^{(A+tK)^n}(f^{[n-1]})\right](S_1(t), \ldots, S_{n-1}(t)) \in \mathcal{S}^p(\mathcal{H})$$
is differentiable at $0$ and in that case $\varphi'(0)=\tilde{\varphi}'(0)$. Let us write, for every $1\leq i \leq n-1$,
$$
S_i(t) = S_i(0) + tS_i'(0) + o_i(t)
$$
where $o_i(t) = o(t)$ depends on $i$. By uniform boundedness of $\Gamma^{(A+tK)^n}(f^{[n-1]}), ~ t\in \mathbb{R}$, and by the multilinearity of operator integrals, we can write
\begin{align*}
\tilde{\varphi}(t)
& = \left[ \Gamma^{(A+tK)^n}(f^{[n-1]}) \right](S_1(0), \ldots, S_{n-1}(0)) \\
& \hspace*{0,5cm} + t \sum_{k=1}^{n-1} \left[ \Gamma^{(A+tK)^n}(f^{[n-1]}) \right](S_1(0),\ldots, S_{k-1}(0), S_k'(0),  S_{k+1}(0), \ldots, S_{n-1}(0)) + \mathcal{O}(t^2).
\end{align*}
By the first part of the proof, $\tilde{\varphi}$ is differentiable at $0$ and by \eqref{deriveeautoadj}, we get the desired formula.
\end{proof}

\begin{crl}\label{FormulaSA3} Let $1 < p < \infty$ and let $n \in\N, n\geq 1$. Let $\mathcal{A} : \mathbb{R} \to \mathcal{B}_{sa}(\mathcal{H})$ be such that $\tilde{\mathcal{A}} : t\in \mathbb{R} \mapsto \mathcal{A}(t) - \mathcal{A}(0) \in \mathcal{S}^p_{sa}(\mathcal{H})$ is $n$-times differentiable in a neighborhood $I$ of $0$. Let $f$ be $n$-times differentiable on $\mathbb{R}$ such that $f^{(n)}$ is bounded. Then, the function
$$
\varphi : t\in \mathbb{R} \mapsto f(\mathcal{A}(t)) - f(\mathcal{A}(0)) \in \mathcal{S}^p(\hilh)
$$
is $n$-times differentiable on $I$ and for every integer $1\leq k  \leq n$ and every $t \in I$,
\begin{equation}\label{Mainformuladerivative2}
\varphi^{(k)}(t)=\sum_{m=1}^{k} \sum_{\substack{l_1,\ldots,l_m\ge 1 \\ l_1+\cdots+l_m=k}}\dfrac{k!}{l_1!\cdots l_m!}\left[\Gamma^{(\mathcal{A}(t))^{m+1}}(f^{[m]})\right]\left(\tilde{\mathcal{A}}^{(l_1)}(t),\ldots,\tilde{\mathcal{A}}^{(l_m)}(t)\right).
\end{equation}
\end{crl}

\begin{proof}
We prove the result by induction on $n$. The case $n=1$ follows from \cite[Theorem 7.13]{KiPoShSu14}. When $n\geq 2$, using Theorem \ref{FormulaSA2} and employing a similar combinatorial reasoning as demonstrated in the proof of \cite[Theorem 5.3.4]{SkToBook} give the result. We leave the details to the reader.
\end{proof}

\begin{rmrk}\label{Corollary42simple}
In the case when $f \in C^n(\mathbb{R})$, the result of Corollary \ref{FormulaSA3} can be proved using \cite[Theorem 3.3]{LeSk20} instead of Theorem \ref{FormulaSA2}.
\end{rmrk}

The next proposition corresponds to the main result of this paper (see Theorem \ref{main}) but with an additional assumption on the function $t\mapsto U(t)$ valued in the unitary operators. We will need it to prove the same result in full generality in Section \ref{sectionmainresult}. The proof makes use of Corollary \ref{FormulaSA3} which is the corresponding result for selfadjoint operators. To do so, we will need the Cayley transform to change the function $t \mapsto f(U(t)) - f(U(0))$ into a function $t \mapsto g(\mathcal{A}(t)) - g(\mathcal{A}(0))$ where $\mathcal{A}$ is valued in the set of selfadjoint operators on $\mathcal{H}$. Denote by $\eta$ the Cayley transform $\eta : \mathbb{R} \to \mathbb{T}\setminus \{1\}$ and by $\eta^{-1}$ its inverse function, defined by
$$\begin{array}{lrcl}
\eta : & \mathbb{R} & \longrightarrow & \mathbb{T}\setminus \{1\} \\
    & x & \longmapsto & \frac{x + i}{x - i}
    \end{array}
    \quad \quad \text{and} \quad \quad
    \begin{array}{lrcl}
\eta^{-1} : &  \mathbb{T}\setminus \{1\} & \longrightarrow & \mathbb{R} \\
    & z & \longmapsto & i\frac{z+1}{z-1}
    \end{array}.
$$
If $A \in \mathcal{B}_{sa}(\hilh)$, then $\eta(A)\in \mathcal{U}(\hilh)$ and $\sigma(\eta(A)) \subset \mathbb{T}\setminus \{1\}$ and conversely, if $U\in \mathcal{U}(\hilh)$ is such that $1\notin \sigma(U)$, then $\eta^{-1}(U)\in \mathcal{B}_{sa}(\hilh)$.

\begin{ppsn}\label{simplecasemainthm}
Let $1<p<\infty$ and $n\in\N$. Let $U:\R\rightarrow\mathcal{U}(\hilh)$ be such that the function $\tilde{U}:t\in\R\mapsto U(t)-U(0)\in\Sp^p(\hilh)$ is $n$-times differentiable on $\mathbb{R}$ and assume that $1\notin \sigma (U(0))$. Let $f: \mathbb{T} \to \mathbb{R}$ be $n$-times differentiable with bounded $n$th derivative $f^{(n)}$. Consider the operator valued function
$$
\varphi : t\mapsto f(U(t)) - f(U(0)) \in \mathcal{S}^p(\hilh).
$$
Then $\varphi$ is $n$-times differentiable in a neighborhood $I$ of $0$ and for every $t \in I$,
\begin{equation}\label{Mainformuladerivative0}
\varphi^{(n)}(t)=\sum_{m=1}^{n} \sum_{\substack{l_1,\ldots,l_m\ge 1 \\ l_1+\cdots+l_m=n}}\dfrac{n!}{l_1!\cdots l_m!}\left[\Gamma^{(U(t))^{m+1}}(f^{[m]})\right]\left(\tilde{U}^{(l_1)}(t),\ldots,\tilde{U}^{(l_m)}(t)\right).
\end{equation}
\end{ppsn}

\begin{proof}
By continuity of $U$, $U(t) \to U(0)$ as $t\to 0$ in the operator norm, and since $1\notin \sigma(U(0))$ and the spectrum is closed, there is a positive number $a>0$ and a real interval $I$ around $0$ such that, for every $t\in I$, $\sigma(U(t)) \subset C_a$ where $C_a := \{ z\in \mathbb{T} \mid |z-1|>a \}$. Note that all the functions of operators and the multiple operator integrals only depend on the values of the associated function on the spectra of the operators. In particular one can extend $\eta^{-1}$ from $C_a$ to a $C^{\infty}$ function on the whole $\mathbb{T}$ if necessary. Let us define, for every $t\in I$, $\mathcal{A}(t)=\eta^{-1}(U(t)) \in \mathcal{B}_{sa}(\hilh)$. Note that for every $t\in I$, $\tilde{\mathcal{A}}(t) := \mathcal{A}(t)-\mathcal{A}(0) \in \mathcal{S}^p(\hilh)$. Indeed, this follows either from \cite[Theorem 2]{AyCoSu16} or by the straightforward identity
$$
\mathcal{A}(t)-\mathcal{A}(0)=-2i\left( U(t)-I \right)^{-1}(U(t)-U(0)) \left( U(0)-I \right)^{-1}
$$
which yields
$$
\| \mathcal{A}(t)-\mathcal{A}(0) \|_p \leq 2 \| \left( U(t)-I \right)^{-1} \| \cdot \|U(t)-U(0)\|_p  \cdot \| \left( U(0)-I \right)^{-1} \|   < +\infty.
$$
Moreover, $\mathcal{A}$ is $n$-times differentiable on $I$. This follows either from \cite[Theorem 3.5]{CCGP} or simply using the fact that $\eta^{-1}$ is a rational function so one can use standards algebraic identities as above.
Let $g : t \in \mathbb{R} \mapsto f(\eta(t))$ and note that
$$
\varphi(t)=f(\eta(\eta^{-1}(U(t))) - f(\eta(\eta^{-1}(U(0))) = g(\mathcal{A}(t)) - g(\mathcal{A}(0)).
$$ The function $g$ is $n$-times differentiable and since $f$ and $\eta$ have bounded derivatives, $g$ has bounded derivatives as well. By Corollary \ref{FormulaSA3}, $\varphi$ is $n$-times differentiable on $I$ and for every $t \in I$,
\begin{align*}
\varphi^{(n)}(t)
& =\sum_{m=1}^{n} \sum_{\substack{l_1,\ldots,l_m\ge 1 \\ l_1+\cdots+l_m=n}}\dfrac{n!}{l_1!\cdots l_m!}\left[\Gamma^{(\mathcal{A}(t))^{m+1}}(g^{[m]})\right]\left(\tilde{\mathcal{A}}^{(l_1)}(t),\ldots,\tilde{\mathcal{A}}^{(l_m)}(t)\right)\\
& :=\sum_{m=1}^{n} \sum_{\substack{l_1,\ldots,l_m\ge 1 \\ l_1+\cdots+l_m=n}} D^{m,l_1,\ldots, l_m}_{g, \mathcal{A}}(t).
\end{align*}
It remains to show that for a fixed $t\in I$,
\begin{align*}
\varphi^{(n)}(t)
& = \sum_{m=1}^{n} \sum_{\substack{l_1,\ldots,l_m\ge 1 \\ l_1+\cdots+l_m=n}}\dfrac{n!}{l_1!\cdots l_m!}\left[\Gamma^{(U(t))^{m+1}}(f^{[m]})\right]\left(\tilde{U}^{(l_1)}(t),\ldots,\tilde{U}^{(l_m)}(t)\right) \\
& \hspace*{0,5cm} :=\sum_{m=1}^{n} \sum_{\substack{l_1,\ldots,l_m\ge 1 \\ l_1+\cdots+l_m=n}} D^{m,l_1,\ldots, l_m}_{f, U}(t).
\end{align*}
To prove this, let $(f_j)_j \subset C^n(\mathbb{T})$ be the sequence given by Lemma \ref{approxdivdiff2}. Then, for every $j\in \mathbb{N}$, the function
$$
\varphi_j : t\mapsto f_j(U(t)) - f(U(0)) \in \mathcal{S}^p(\hilh)
$$
is $n$-times differentiable on $I$ and
\begin{equation}\label{fromunitosa}
\sum_{m=1}^{n} \sum_{\substack{l_1,\ldots,l_m\ge 1 \\ l_1+\cdots+l_m=n}} D^{m,l_1,\ldots, l_m}_{f_j, U}(t) \labelrel={myeq:equality1} \varphi_j^{(n)}(t) \labelrel={myeq:equality2} \sum_{m=1}^{n} \sum_{\substack{l_1,\ldots,l_m\ge 1 \\ l_1+\cdots+l_m=n}} D^{m,l_1,\ldots, l_m}_{f_j \circ \eta, \mathcal{A}}(t).
\end{equation}
Indeed, since $f_j \in C^n(\mathbb{T})$, the equality \eqref{myeq:equality1} comes from Theorem \cite[Theorem 3.5]{CCGP}, while the equality \eqref{myeq:equality2} follows from the same computations performed for $f$ in the first part of the proof.

Fix $1\leq m \leq n$ and let $l_1,\ldots,l_m \geq 1$ be such that $l_1+\cdots + l_m = n$. The assumptions on $(f_j)_j$ ensure that $(f_j^{[m]})_j$ is $w^*$-convergent to $f^{[m]}$ for the $w^*$-topology of $L^{\infty}(\lambda_{U(t)} \times \cdots \times \lambda_{U(t)})$. By Lemma \ref{LemmeUB}, it follows that
$$
D^{m,l_1,\ldots, l_m}_{f_j, U}(t) \underset{j\to +\infty}{\longrightarrow} D^{m,l_1,\ldots, l_m}_{f, U}(t)
$$
weakly in $\mathcal{S}^p(\hilh)$. On the other hand, by \cite[Lemma 2.3]{PoSkSu15}, we have, for every $(\lambda_1, \ldots, \lambda_{m+1})\in \mathbb{R}^{m+1}$,
\begin{align*}
(f_j \circ \eta)^{[m]}(\lambda_1, \ldots, \lambda_{m+1})
& =\sum_{k=1}^m \sum_{1=i_0<\cdots<i_k=m+1} \frac{(-1)^{k+1}i^{m-k+1}}{2^{m-k+1}} f_j^{[k]}(\eta(\lambda_{i_0}),\ldots, \eta(\lambda_{i_k})) \\
& \hspace*{0,5cm} \times \prod_{j=1}^{k-1}(\eta(\lambda_{i_j})-1)^2 \prod_{l\in \{1,\ldots,m+1\} \setminus \{i_1,\ldots, i_{k-1} \} } (\eta(\lambda_l)-1).
\end{align*}
This formula holds true as well if $f_j$ is replaced by $f$, with the same proof. In particular, the pointwise convergence of $f_j^{[k]}$ to $f^{[k]}$ implies the pointwise convergence of $((f_j \circ \eta)^{[m]})_j$ to $(f \circ \eta)^{[m]}=g^{[m]}$. Together with the boundedness of each $(f_j^{[k]})_j$ and hence the boundedness of $((f_j \circ \eta)^{[m]})_j$, we get the $w^*$-convergence of $((f_j \circ \eta)^{[m]})_j$ to $g^{[m]}$ for the $w^*$-topology of $L^{\infty}(\lambda_{\mathcal{A}(t)} \times \cdots \times \lambda_{\mathcal{A}(t)})$. By Lemma \ref{LemmeUB} and the paragraph preceding it, it follows that
$$
D^{m,l_1,\ldots, l_m}_{f_j\circ \eta, \mathcal{A}}(t) \underset{j\to +\infty}{\longrightarrow} D^{m,l_1,\ldots, l_m}_{g, U}(t)
$$
weakly in $\mathcal{S}^p(\hilh)$. Finally, after taking the limit as $j\to +\infty$ in the weak topology of $\mathcal{S}^p(\hilh)$ in \eqref{fromunitosa}, we obtain that
$$
\sum_{m=1}^{n} \sum_{\substack{l_1,\ldots,l_m\ge 1 \\ l_1+\cdots+l_m=n}} D^{m,l_1,\ldots, l_m}_{f, U}(t) = \sum_{m=1}^{n} \sum_{\substack{l_1,\ldots,l_m\ge 1 \\ l_1+\cdots+l_m=n}} D^{m,l_1,\ldots, l_m}_{g, \mathcal{A}}(t),
$$
which gives the desired formula for $\varphi^{(n)}(t)$ and concludes the proof.
\end{proof}

\begin{rmrk}\label{okspectrumdiffT}
Proposition \ref{simplecasemainthm} holds true as well if we simply assume that $\sigma(U(0)) \neq \mathbb{T}$. Indeed, by picking $e^{i\theta} \notin \sigma(U(0))$ and changing the function $U$ into $e^{-i\theta}U$ (in that case, $1\notin \sigma(e^{-i\theta}U(0))$) and $f$ into $h(z) = f(e^{i\theta}z)$ we get that $\varphi : t\in \mathbb{R} \mapsto f(U(t))-f(U(0))=h(e^{-i\theta}U(t))- h(e^{-i\theta}U(0))$ so that $\varphi$ is differentiable in a neighborhood of $0$. Moreover, it is easy to check that $h^{[n]}(\lambda_1, \ldots, \lambda_{n+1}) = e^{in\theta} f^{[n]}(e^{i\theta}\lambda_1, \ldots, e^{i\theta}\lambda_{n+1})$ and $(e^{-i\theta}\tilde{U})^{(l_1)}(t) = e^{-i\theta}\tilde{U}^{(l_1)}(t)$ so that
\begin{align*}
\varphi^{(k)}(t)
& =\sum_{m=1}^{k} \sum_{\substack{l_1,\ldots,l_m\ge 1 \\ l_1+\cdots+l_m=k}}\dfrac{k!}{l_1!\cdots l_m!} \left[\Gamma^{(e^{-i\theta}U(t))^{m+1}}(h^{[m]})\right]\left(e^{-i\theta}\tilde{U}^{(l_1)}(t),\ldots,e^{-i\theta}\tilde{U}^{(l_m)}(t)\right) \\
& =\sum_{m=1}^{k} \sum_{\substack{l_1,\ldots,l_m\ge 1 \\ l_1+\cdots+l_m=k}}\dfrac{k!}{l_1!\cdots l_m!} \left[\Gamma^{(e^{-i\theta}U(t))^{m+1}}((f^{[m]})_r)\right]\left(\tilde{U}^{(l_1)}(t),\ldots,\tilde{U}^{(l_m)}(t)\right),
\end{align*}
where $(f^{[m]})_r(\lambda_1, \ldots, \lambda_{m+1}) = f^{[m]}(e^{i\theta}\lambda_1, \ldots, e^{i\theta}\lambda_{m+1})$. It is now easy to check (it comes from the construction of multiple operator integrals) that
$$
\Gamma^{(e^{-i\theta}U(t))^{m+1}}((f^{[m]})_r) = \Gamma^{(U(t))^{m+1}}(f^{[m]}).
$$
\end{rmrk}

\section{$\mathcal{S}^p$-differentiability for non continuously differentiable functions}\label{sectionmainresult}

\noindent In this section, we will prove the following which is the main result of this paper.

\begin{thm}\label{main}
Let $1<p<\infty$ and $n\in\N$. Let $U:\R\rightarrow\mathcal{U}(\hilh)$ be such that the function $\tilde{U}:t\in\R\mapsto U(t)-U(0)\in\Sp^p(\hilh)$ is $n$-times differentiable on $\mathbb{R}$. Let $f: \mathbb{T} \to \mathbb{R}$ be $n$-times differentiable with bounded $n$th derivative $f^{(n)}$. Consider the operator valued function
$$
\varphi : t\mapsto f(U(t)) - f(U(0)) \in \mathcal{S}^p(\hilh).
$$
Then $\varphi$ is $n$-times differentiable on $\mathbb{R}$ and for every integer $1\leq k  \leq n$ and every $t \in \mathbb{R}$,
\begin{equation}\label{Mainformuladerivative}
\varphi^{(k)}(t)=\sum_{m=1}^{k} \sum_{\substack{l_1,\ldots,l_m\ge 1 \\ l_1+\cdots+l_m=k}}\dfrac{k!}{l_1!\cdots l_m!}\left[\Gamma^{(U(t))^{m+1}}(f^{[m]})\right]\left(\tilde{U}^{(l_1)}(t),\ldots,\tilde{U}^{(l_m)}(t)\right).
\end{equation}
\end{thm}

\begin{rmrk} It follows from the boundedness of multiple operator integrals given by Theorem \ref{MOISpBound} that if $\tilde{U}$ has bounded derivatives, then $\varphi$ has bounded derivatives as well.
\end{rmrk}

\begin{rmrk}\label{Taylorform} \begin{enumerate}
\item Once the formula for the derivatives of $\varphi$ has been established, it is easy to check, by induction, that the operator Taylor remainder defined by 
\begin{align}
\label{ref20}R_{n,f,U}(t):=f(U(t))-f(U(0))-\sum_{k=1}^{n-1} \dfrac{1}{k!}\varphi^{(k)}(0)
\end{align} 
satisfies, for any $t\in\R$,
\begin{equation}\label{TaylorFormula1}
R_{n,f,U}(t)=\sum_{m=1}^n\sum_{\substack{l_1,\ldots,l_m \ge 1 \\ l_1+\cdots+l_m=n}}\left[\Gamma^{U(t),(U(0))^{m}}(f^{[m]})\right]\left(R_{l_1,U}(t),\dfrac{\tilde{U}^{(l_2)}(0)}{l_2!},\ldots,\dfrac{\tilde{U}^{(l_m)}(0)}{l_m!}\right),
\end{equation}
where $R_{1,U}(t):=\tilde{U}(t)$ and for any $l_1\ge 2$,
\begin{equation*}
R_{l_{1},U}(t):=\tilde{U}(t)-\sum_{k=1}^{l_{1}-1} \dfrac{1}{k!}\tilde{U}^{(k)}(0).
\end{equation*}
We refer to the proof of \cite[Proposition 3.5 (ii)]{CCGP} for more details and references.
\item Theorem \ref{main} applies in particular to the function $U(t) = e^{itA}U$ with $U\in \mathcal{U}(\hilh)$ and $A\in \Sp^p_{sa}(\hilh)$, and we retrieve \cite[Corollary 3.6]{CCGP} in the more general case of a function $f$ with (non necessarily continuous) bounded $n$th derivative. In particular, with the same proof than that of \cite[Corollary 3.6]{CCGP}, we obtain that
\begin{align}\label{introremainderest}
\|R_{n,f,U}(1)\|_{\frac{p}{n}}\le \tilde{c}_{p,n}\sum_{m=1}^{n}\|f^{(m)}\|_{\infty}\|A\|_p^n,
\end{align}
where $\tilde{c}_{p,n}$ is a positive constant depending on $p$ and $n$. 
\end{enumerate}
\end{rmrk}

To prove Theorem \ref{main}, we will carefully approximate, on a subspace of the Hilbert space $\hilh$, the unitary operator $U(0)$ by another unitary whose spectrum is not the whole $\mathbb{T}$, in order to use Proposition \ref{simplecasemainthm}. The relevant definitions and the first properties of the approximation are given in Subsection \ref{subsecappruni}. The key auxiliary results from Lemma \ref{intermlemma} will detail the regularity of this approximation process.

\subsection{Approximation of unitaries}\label{subsecappruni}

Let $V\in \mathcal{U}(\hilh)$. Define, for every $j\geq 1$, $A_j:=\{e^{2i\pi t} \mid 0\leq t \leq \frac{j-1}{j} \}$ and set $P_j:=E^V(A_j)$. Then $(P_j)_{j\geq 1}$ is an increasing sequence of selfadjoint projections which converges strongly to $I_{\mathcal{H}}$. Recall that this implies that for every $K \in \mathcal{S}^p(\hilh)$, $P_j X$, $XP_j$ and $P_jXP_j$ converge to $X$ in $\mathcal{S}^p$ as $j\to +\infty$. Moreover, $P_j$ commutes with $V$ and the operator $V_j := P_jVP_j = P_jV = VP_j$ is unitary on the Hilbert space $\mathcal{H}_j := P_j \mathcal{H}$ and its spectrum satisfies $\sigma(V_j) \subset A_j$.

Note that if $K\in \mathcal{S}^p(\hilh)$, $P_j K P_j \in \mathcal{S}^p(\hilh)$ with $\| P_j K P_j \|_p \leq \| K \|_p$ and we can see $P_j K P_j$ as an element of $\mathcal{S}^p(\hilh_j)$. Similarly, if $X\in \mathcal{S}^p(\mathcal{H}_j)$, we can extend $X$ on $\hilh$ and keep denoting this operator by $X$, and in that case $P_jXP_j = X_{| \hilh_M} \oplus 0_{\mathcal{H}_M^{\perp}}$.

\begin{ppsn}\label{Approxunit}
Let $1<p<\infty$. Let $A\in \mathcal{S}^p_{sa}(\hilh)$ and define $A_j:=P_jAP_j$. Let $n \geq 2$ be an integer and let $f : \mathbb{T} \to \mathbb{C}$ be $n$-times differentiable such that $f^{(n)}$ is bounded.
\begin{enumerate}
\item For every $K_1,\ldots, K_{n-1}\in \mathcal{S}^p(\hilh)$ and every $j\in \mathbb{N}$,
\begin{align*}
\left[ \Gamma^{(e^{iA_j}V_j)^{n}}(f^{[n-1]}) \right](K_{1,j}, \ldots, K_{n-1,j})  = \left[ \Gamma^{(e^{iA_j}V)^{n}}(f^{[n-1]})  \right](K_{1,j}, \ldots, K_{n-1,j}),
\end{align*}
where $K_{i,j} := P_jK_iP_j$.
\item For every $K_1,\ldots, K_n\in \mathcal{S}^p(\hilh)$,
$$
\left[ \Gamma^{(VP_j)^{n+1}}(f^{[n]}) \right](K_{1,j},\ldots, K_{n,j})  \overset{\| \cdot  \|_p}{\underset{j\to +\infty}{\longrightarrow}} \left[ \Gamma^{(V)^{n+1}}(f^{[n]}) \right](K_1,\ldots, K_n).
$$
\end{enumerate}
\end{ppsn}

\begin{proof}
Let us prove $(1)$. Recall that since $f^{[n-1]}$ is bounded, we have, by construction, $\Gamma^{(e^{iA_j}V_j)^{n}}(f^{[n-1]}) \in \mathcal{B}_{n-1}(\mathcal{S}^2(\hilh))$. Let us first etablish the formula when $K_1, \ldots, K_{n-1} \in \mathcal{S}^2(\hilh)$. Since $f^{[n-1]}$ is continuous on $\mathbb{T}^n$ it is sufficient, by a simple approximation argument, to prove the formula when $f^{[n-1]}$ is replaced by a trigonometric polynomial $\varphi$ on $\mathbb{T}^n$ and by linearity, we can assume that $\varphi = f_1 \otimes \cdots \otimes f_n$, where for any $1\leq i \leq n, f_i$ is a trigonometric polynomial on $\mathbb{T}$. Since $P_j$ commutes with $e^{iA_j}$ and with $V$, it a easy to check that for every $1\leq i \leq n$,
$$
f_i(e^{iA_j}V_j) P_j = P_jf_i(e^{iA_j}V_j)  = P_jf_i(e^{iA_j}V) = f_i(e^{iA_j}V)P_j.
$$
It follows that
\begin{align*}
& \left[ \Gamma^{(e^{iA_j}V_j)^{n}}(\varphi) \right](K_{1,j}, \ldots, K_{n-1,j})  \\
& \hspace*{0,5cm} = f_1(e^{iA_j}V_j)P_jK_1P_jf_2(e^{iA_j}V_j)P_jK_2P_j \cdots P_jK_{n-1}P_jf_n(e^{iA_m}V_j) \\
& \hspace*{0,5cm} = f_1(e^{iA_j}V)P_jK_1P_jf_2(e^{iA_j}V)P_jK_2P_j \cdots P_jK_{n-1}P_jf_n(e^{iA_j}V) \\
& \hspace*{0,5cm} =  \left[ \Gamma^{(e^{iA_j}V)^{n}}(\varphi) \right](K_{1,j}, \ldots, K_{n-1,j}).
\end{align*}
This proves the formula for $\varphi = f^{[n-1]}$ and $K_1, \ldots, K_{n-1} \in \mathcal{S}^2(\hilh)$. In particular, the formula holds true when $K_i \in \mathcal{S}^2 \cap \mathcal{S}^p$, and approximating (in the $\mathcal{S}^p$-norm) any $K_i\in \mathcal{S}^p$ by a sequence of elements of $\mathcal{S}^2 \cap \mathcal{S}^p$ and using the fact that $\Gamma^{(e^{iA_j}V_j)^{n}}(f^{[n-1]}), \Gamma^{(e^{iA_j}V)^{n}}(f^{[n-1]}) \in \mathcal{B}_{n-1}(\mathcal{S}^p(\hilh))$, we obtain the desired formula.

For the proof of $(2)$, we only make some with minor changes: we first etablish that for every $K_1, \ldots, K_n \in \mathcal{S}^2(\hilh)$,
\begin{equation}\label{cutprojformula}
\left[ \Gamma^{(VP_j)^{n+1}}(f^{[n]}) \right](K_{1,j},\ldots, K_{n,j}) = \left[ \Gamma^{(V)^{n}}(f^{[n]})  \right](K_{1,j},\ldots, K_{n,j}).
\end{equation}
By the $w^{\ast}$-continuity of multiple operator integrals (see the paragraph before Definition \ref{DefMOI}), it is sufficient to prove the identity when $f^{[n]}$ is replaced by $\varphi \in L^{\infty}(\lambda_{V})\otimes\cdots\otimes L^{\infty}(\lambda_V)$, and by linearity, we can further assume $\varphi =f_1 \otimes \cdots \otimes f_{n+1}$, where for any $1\leq i \leq n+1$, $f_i\in L^{\infty}(\lambda_{V})$. Note that $VP_j = V \chi_{A_j}(V) = g_i(V)$ where $g_i(x)=x\chi_{A_j}(x)$ and it is straightforward to check that $(f_i \circ g_i) \chi_{A_j} = f_i \chi_{A_j}$. By \cite[Corollary 5.6.29]{KadRin}, we have
$$
f_i(VP_j)P_j  = (f_i \circ g_i)(V) \chi_{A_j}(V) = ((f_i\circ g_i)  \chi_{A_j})(V) = f_i(V)\chi_{A_j}(V)=f_i(V)P_j,
$$
and similarly, $P_jf_i(VP_j)=P_jf_i(V)$. The same computations performed to prove $(1)$ show that \ref{cutprojformula} hold true. Moreover, this formula extends, as before, when $K_1, \ldots, K_n \in \mathcal{S}^p(\hilh)$. Finally, the fact that $K_{i,j} \to K_i$ in $\mathcal{S}^p(\hilh)$ for every $1\leq i \leq  n$ together with $\Gamma^{(V)^{n+1}}(f^{[n]})\in \mathcal{B}_n(\mathcal{S}^p(\hilh))$ yield
\begin{align*}
\left[ \Gamma^{(VP_j)^{n+1}}(f^{[n]}) \right](K_{1,j},\ldots, K_{n,j})
& = \left[ \Gamma^{(V)^{n+1}}(f^{[n]})  \right](K_{1,j},\ldots, K_{n,j}) \\
& \underset{m\to +\infty}{\longrightarrow} \left[ \Gamma^{(V)^{n+1}}(f^{[n]})  \right](K_1,\ldots, K_n)
\end{align*}
in $\mathcal{S}^p(\hilh)$, which concludes the proof.
\end{proof}

\subsection{Proof of the main result}

In this subsection, we will prove Theorem \ref{main}. First of all, we need the following Lemma. We postpone its proof at the end of the paper to avoid repeating certain arguments and computations which, for some of them, will be very similar to those in the proof of Theorem \ref{main}.

\begin{lma}\label{intermlemma} Let $1<p<\infty$ and $n\in \mathbb{N}$. Let $V\in \mathcal{U}(\hilh)$. Let $\mathcal{A} : \mathbb{R} \to \mathcal{S}^p_{sa}(\hilh)$ be $n$-times differentiable on $\mathbb{R}$ with $\mathcal{A}(0)=0$. Define
$$
U(t):=e^{i\mathcal{A}(t)}V \quad \text{and} \quad \tilde{U}(t):=e^{i\mathcal{A}(t)}V - V \in \mathcal{S}^p(\hilh),
$$
and, for every $j\in \mathbb{N}$,
$$\mathcal{A}_j(t) := P_j \mathcal{A}(t)P_j, \quad U_j(t):=e^{i\mathcal{A}_j(t)}V_j \quad \text{and} \quad \widetilde{U_j}(t):=U_j(t)-U_j(0) \in \mathcal{S}^p(\hilh),$$
where $P_j$ and $V_j=VP_j$ are defined at the beginning of Subsection \ref{subsecappruni}.
Then, we have the following properties: 
\begin{enumerate}
\item For every $\epsilon > 0$, there exists $J\in \mathbb{N}$ and $\alpha>0$ such that
\begin{equation}\label{AmtoA1}
\forall \, j\geq J, \, \forall \, |t| < \alpha, \ \left\| e^{i\mathcal{A}(t)}V - e^{i\mathcal{A}_j(t)}V \right\|_p \leq \epsilon |t|.
\end{equation}
\item There exist $\alpha>0$ and a constant $C>0$ such that
\begin{equation}\label{AmtoA2}
\forall \, j\in \mathbb{N}, \, \forall \, |t| < \alpha, \ \left\| e^{i\mathcal{A}(t)}V - V \right\|_p \leq C |t| \quad \text{and} \quad \left\| e^{i\mathcal{A}_j(t)}V - V \right\|_p \leq C |t|.
\end{equation}
\item For every $j\in \mathbb{N}$, $\tilde{U}$ and $\widetilde{U_j}$ are $n$-times differentiable on $\mathbb{R}$ and for every $0\leq k \leq n$ and every $t\in \mathbb{R}$,
\begin{equation}\label{AmtoA3}
P_j\widetilde{U_j}^{(k)}(t)P_j = \widetilde{U_j}^{(k)}(t).
\end{equation}
\item For every $\epsilon > 0$, there exists $J\in \mathbb{N}$ and $\alpha>0$ such that, for every $1\leq k \leq n-1$,
\begin{equation}\label{AmtoA4}
\forall \, j\geq J, \, \forall \, |t| < \alpha, \ \| \tilde{U}^{(k)}(t) - \widetilde{U_j}^{(k)}(t) \|_p \leq \epsilon.
\end{equation}
\item Let $0 \leq k \leq n-1$. Then we can write
$$
\tilde{U}^{(k)}(t) = \tilde{U}^{(k)}(0) + t R_{k}(t) \quad \text{and} \quad \widetilde{U_j}^{(k)}(t) = \widetilde{U_j}^{(k)}(0) + t R^j_{{k}}(t),
$$
where, for every $\epsilon > 0$, there exist $J\in \mathbb{N}$ and $\alpha > 0$ such that
\begin{equation}\label{AmtoA5}
\forall \, j\geq J, \, \forall \, |t| < \alpha, \ \left\| R_k(t)-R_k^j(t)  \right\|_p \leq \epsilon.
\end{equation}
\end{enumerate}
\end{lma}

\begin{proof}[Proof of Theorem \ref{main}] The proof will be divided into three steps. First, we show that we can rewrite the function $U$ in a convenient way. Next, we will approximate the unitary $U(0)$ in order to use Proposition \ref{simplecasemainthm}. Finally, thanks to several different estimates that will be using Lemma \ref{intermlemma}, we will obtain the result.\\

\textbf{Step 1. Simplification of the function U.} First of all, note that by translation, it is sufficient to prove the result for $t=0$. Let $V:=U(0)\in \mathcal{U}(\hilh)$. By continuity of $U$, $U(t)V^* \to I$ as $t\to 0$ in the operator norm, so that, for $t\in I$ where $I$ is a real interval centered at $t=0$, $\|U(t)V^* - I \| < \frac{1}{2}$. In particular, we can set $\mathcal{A}(t) := -i \log(U(t)V^*)$ and we get that $\mathcal{A}(t) \in \mathcal{B}_{sa}(\hilh)$. This function satisfies $\mathcal{A}(0)=0$ and $e^{i\mathcal{A}(t)}V = U(t)$. Moreover, the assumption $U(t) - U(0) \in \mathcal{S}^p(\hilh)$ implies that $\mathcal{A}(t)\in \mathcal{S}^p_{sa}(\hilh)$ and since $\mathcal{A}(t)$ is obtained by means of a power series, the fact that $\tilde{U} : \mathbb{R} \to \mathcal{S}^p(\hilh)$ is $n$-times differentiable on $\mathbb{R}$ implies that $\mathcal{A} : I \to \mathcal{S}^p_{sa}(\hilh)$ is $n$-times differentiable on $I$. Alternatively, since $\log$ is $C^{\infty}$ in a neighborhood of $1$, this follows as well from \cite[Theorem 3.5]{CCGP}. Hence, from now on, we will assume that
$$\forall \, t\in I, \ U(t)=e^{i\mathcal{A}(t)}V,$$
where $\mathcal{A}$ has the properties given above.\\

\textbf{Step 2. Initiation of the approximation process.} Define, for every $j\geq 1$, $A_j:=\{e^{2i\pi t} \mid 0\leq t \leq \frac{j-1}{j} \}$ and set $P_j:=E^V(A_j)$, as in Subsection \ref{subsecappruni}. Recall that the operator $V_j := P_jVP_j  = P_jV=VP_j$ is unitary on the Hilbert space $\mathcal{H}_j := P_j \mathcal{H}$ and its spectrum satisfies $\sigma(V_j) \subset A_j$. Define
$$\mathcal{A}_j(t) := P_j \mathcal{A}(t)P_j, \quad U_j(t):=e^{i\mathcal{A}_j(t)}V_j \quad \text{and} \quad \widetilde{U_j}(t):=U_j(t)-U_j(0) \in \mathcal{S}^p,$$
where $\mathcal{A}_j(t)$ and $\tilde{U_j}(t)$ can be seen either as elements of $\mathcal{S}^p(\mathcal{H}_j)$ or $\mathcal{S}^p(\mathcal{H})$.

Now, define
$$
\varphi_j : t\in \mathbb{R} \mapsto f(U_j(t)) - f(U_j(0)) \in \mathcal{S}^p(\mathcal{H}_j).
$$
The operator $\varphi_j(t)$ acts as well on $\mathcal{H}$ and is equal to $0$ on $\mathcal{H}_j^{\perp}$. Since $e^{i\mathcal{A}_j(t)}V_j \in \mathcal{U}(\mathcal{H}_j)$ and $\sigma(V_j) \subset A_j$ and hence $\sigma(V_j) \neq \mathbb{T}$, by Proposition \ref{simplecasemainthm} and Remark \ref{okspectrumdiffT}, $\varphi_j$ is $n$-times differentiable in a neighborhood of $0$, which we can assume to be equal to $I$, so that we can write, for every $t\in I$,
\begin{equation}\label{diffphim}
\varphi_j^{(n-1)}(t) - \varphi_j^{(n-1)}(0) - t\varphi_j^{(n)}(0) = o_j(t),
\end{equation}
where $o_j(t)=o(t)$ depends on $j$, and where $\varphi_j^{(n-1)}$ and $\varphi_j^{(n)}$ are given by Formula \eqref{Mainformuladerivative}.

Since $f \in C^{n-1}(\mathbb{T})$, by \cite[Theorem 3.5]{CCGP}, $\varphi$ is $(n-1)$-times differentiable on $\mathbb{R}$ and for every $t\in \mathbb{R}$,
\begin{align*}
\varphi^{(n-1)}(t)=\sum_{m=1}^{n-1} \sum_{\substack{l_1,\ldots,l_m\ge 1 \\ l_1+\cdots+l_m=n-1}}\dfrac{(n-1)!}{l_1!\cdots l_m!}\left[\Gamma^{(U(t))^{m+1}}(f^{[m]})\right]\left(\tilde{U}^{(l_1)}(t),\ldots,\tilde{U}^{(l_m)}(t)\right).
\end{align*}
Let us define
\begin{align*}
T := \sum_{m=1}^{n} \sum_{\substack{l_1,\ldots,l_m\ge 1 \\ l_1+\cdots+l_m=n}}\dfrac{n!}{l_1!\cdots l_m!}\left[\Gamma^{(V)^{m+1}}(f^{[m]})\right]\left(\tilde{U}^{(l_1)}(0),\ldots,\tilde{U}^{(l_m)}(0)\right)
\end{align*}
and
\begin{align*}
T_j := \sum_{m=1}^{n} \sum_{\substack{l_1,\ldots,l_m\ge 1 \\ l_1+\cdots+l_m=n}}\dfrac{n!}{l_1!\cdots l_m!}\left[\Gamma^{(V_j)^{m+1}}(f^{[m]})\right]\left(\widetilde{U_j}^{(l_1)}(0),\ldots,\widetilde{U_j}^{(l_m)}(0)\right).
\end{align*}
In particular $\varphi_j^{(n)}(0) = T_j$. To prove the Theorem, we have to show that
\begin{equation}\label{toprovediff}
\varphi^{(n-1)}(t) - \varphi^{(n-1)}(0) - tT = o(t)
\end{equation}
as $t\to 0$. Note that if $n=1$, we do not use \cite[Theorem 3.5]{CCGP} and \eqref{toprovediff} reduces to
$$
f(U(t))-f(V) - t\left[\Gamma^{V,V}(f^{[1]})\right](\tilde{U}'(0)) = o(t).
$$
To prove our claim, let us write, for every $j\in \mathbb{N}$,
\begin{align}
& \nonumber \varphi^{(n-1)}(t) - \varphi^{(n-1)}(0) - tT \\
&\label{reformulate1} \hspace*{0,5cm} = L_{j}(t) - t(T-T_j) + \left( \varphi_j^{(n-1)}(t) - \varphi_j^{(n-1)}(0) - t\varphi_j^{(n)}(0) \right),
\end{align}
where
$$
L_{j}(t) := \varphi^{(n-1)}(t) - \varphi_j^{(n-1)}(t) + \varphi_j^{(n-1)}(0) - \varphi^{(n-1)}(0).
$$
First, we will estimate the quantity $T-T_j$ uniformly for $j$ large enough, and secondly, we will estimate the term $L_j(t)$ for $t$ small enough and $j$ large enough. Eventually, we will use \eqref{diffphim} to estimate the last term appearing in \eqref{reformulate1}, for a fixed integer $j$. \\

\textbf{Step 3. Estimates in the approximation process.} Let us fix $\epsilon >0$.

\noindent \underline{Estimate of $T-T_j$}. Let $1\leq m \leq n$ and let $l_1,\ldots,l_m\ge 1$ be such that $l_1+\cdots+l_m=n$. According to Proposition \ref{Approxunit} $(2)$, if $j\geq J^1$ is large enough,
\begin{align*}
& \left\| \left[\Gamma^{(V)^{m+1}}(f^{[m]})\right]\left(\tilde{U}^{(l_1)}(0),\ldots,\tilde{U}^{(l_m)}(0)\right) \right. \\
& \hspace*{1,5cm} \left. - \left[\Gamma^{(V_j)^{m+1}}(f^{[m]})\right]\left(P_j\tilde{U}^{(l_1)}(0)P_j,\ldots,P_j \tilde{U}^{(l_m)}(0)P_j\right) \right\|_p \\
& \leq \epsilon,
\end{align*}
and, according to Lemma \ref{intermlemma} $(3)$ and $(4)$, and to the uniform boundedness of $\left( \Gamma^{(V_j)^{m+1}}(f^{[m]}) \right)_j$, if $j\geq J^2$ is large enough,
\begin{align*}
& \left\| \left[\Gamma^{(V_j)^{m+1}}(f^{[m]})\right]\left(P_j\tilde{U}^{(l_1)}(0)P_j,\ldots,P_j \tilde{U}^{(l_m)}(0)P_j\right) \right. \\
& \hspace*{1,5cm} \left. - \left[\Gamma^{(V_j)^{m+1}}(f^{[m]})\right]\left(\widetilde{U_j}^{(l_1)}(0),\ldots,\widetilde{U_j}^{(l_m)}(0)\right) \right\|_p \\
& \leq \epsilon.
\end{align*}
It follows that for every $j\geq \max \{ J^1,J^2 \} =: J$,
\begin{align*}
& \left\| \left[\Gamma^{(V)^{m+1}}(f^{[m]})\right]\left(\tilde{U}^{(l_1)}(0),\ldots,\tilde{U}^{(l_m)}(0)\right) - \left[\Gamma^{(V_j)^{m+1}}(f^{[m]})\right]\left(\widetilde{U_j}^{(l_1)}(0),\ldots,\widetilde{U_j}^{(l_m)}(0)\right) \right\|_p \leq 2\epsilon.
\end{align*}
Hence, since $T$ and $T_j$ are finite sums of such terms, there exist a constant $c_0$ and an integer $J_0$ such that
\begin{equation}\label{estimate3}
\forall \, j \geq J_0, \quad \|T-T_j\|_p \leq c_0 \epsilon.
\end{equation}

\noindent \underline{Estimate of $L_j(t)$}. Recall that
\begin{align*}
\varphi_j^{(n-1)}(t)
& = \sum_{m=1}^{n-1} \sum_{\substack{l_1,\ldots,l_m\ge 1 \\ l_1+\cdots+l_m=n-1}}\dfrac{(n-1)!}{l_1!\cdots l_m!}\left[\Gamma^{(U_j(t))^{m+1}}(f^{[m]})\right]\left(\widetilde{U_j}^{(l_1)}(t),\ldots,\widetilde{U_j}^{(l_m)}(t)\right) \\
& = \sum_{m=1}^{n-1} \sum_{\substack{l_1,\ldots,l_m\ge 1 \\ l_1+\cdots+l_m=n-1}}\dfrac{(n-1)!}{l_1!\cdots l_m!}\left[\Gamma^{(e^{i\mathcal{A}_j(t)}V)^{m+1}}(f^{[m]})\right]\left(\widetilde{U_j}^{(l_1)}(t),\ldots,\widetilde{U_j}^{(l_m)}(t)\right),
\end{align*}
where the last equality follows from Lemma \ref{intermlemma} $(3)$ and Proposition \ref{Approxunit} $(1)$. When $t=0$, $e^{i\mathcal{A}_j(0)}V = V$ so we have
$$
\varphi_j^{(n-1)}(0) = \sum_{m=1}^{n-1} \sum_{\substack{l_1,\ldots,l_m\ge 1 \\ l_1+\cdots+l_m=n-1}}\dfrac{(n-1)!}{l_1!\cdots l_m!}\left[\Gamma^{(V)^{m+1}}(f^{[m]})\right]\left(\widetilde{U_j}^{(l_1)}(0),\ldots,\widetilde{U_j}^{(l_m)}(0)\right).
$$
Fix $1\leq m \leq n-1$ and $l_1,\ldots,l_m\ge 1$ such that $l_1+\cdots+l_m=n-1$. Since 
$$
L_j(t) = \varphi^{(n-1)}(t) - \varphi_j^{(n-1)}(t) + \varphi_j^{(n-1)}(0) - \varphi^{(n-1)}(0),
$$
according to the latter and by linearity, if we show that
\begin{align*}
& R_j(t) \\
& := \left[\Gamma^{(e^{i\mathcal{A}(t)}V)^{m+1}}(f^{[m]})\right]\left(\tilde{U}^{(l_1)}(t),\ldots,\tilde{U}^{(l_m)}(t)\right) - \left[\Gamma^{(e^{i\mathcal{A}_j(t)}V)^{m+1})}(f^{[m]})\right]\left(\widetilde{U_j}^{(l_1)}(t),\ldots,\widetilde{U_j}^{(l_m)}(t)\right)\\
& \hspace*{1cm} + \left[\Gamma^{(V)^{m+1}}(f^{[m]})\right]\left(\widetilde{U_j}^{(l_1)}(0),\ldots,\widetilde{U_j}^{(l_m)}(0)\right) - \left[\Gamma^{(V)^{m+1}}(f^{[m]})\right]\left(\tilde{U}^{(l_1)}(0),\ldots,\tilde{U}^{(l_m)}(0)\right)
\end{align*}
satisfies $\| R_j(t) \|_p \leq c \epsilon |t|$ for some constant $c$ depending only on $p, m,l_1,\ldots, l_m, f$ and where $j$ is large enough and $t$ small enough, a similar inequality will hold true for $\|L_j(t)\|_p$.

Let us write
\begin{align*}
R_j(t) = S_{1,j}(t) + S_{2,j}(t) + S_{3,j}(t),
\end{align*}
where
$$
S_{1,j}(t) = \left[\Gamma^{(e^{i\mathcal{A}(t)}V)^{m+1}}(f^{[m]}) - \Gamma^{(e^{i\mathcal{A}_j(t)}V)^{m+1}}(f^{[m]})\right]\left(\tilde{U}^{(l_1)}(t),\ldots,\tilde{U}^{(l_m)}(t)\right),
$$
\begin{align*}
S_{2,j}(t)
& = \left[\Gamma^{(e^{i\mathcal{A}_j(t)}V)^{m+1}}(f^{[m]}) - \Gamma^{(V)^{m+1}}(f^{[m]}) \right]\left(\tilde{U}^{(l_1)}(t),\ldots,\tilde{U}^{(l_m)}(t)\right) \\
& \hspace*{0.5cm} + \left[\Gamma^{(V)^{m+1}}(f^{[m]})\right]\left(\widetilde{U_j}^{(l_1)}(t),\ldots,\widetilde{U_j}^{(l_m)}(t)\right) \\
& \hspace*{0.5cm} - \left[\Gamma^{(e^{i\mathcal{A}_j(t)}V)^{m+1}}(f^{[m]})\right]\left(\widetilde{U_j}^{(l_1)}(t),\ldots,\widetilde{U_j}^{(l_m)}(t)\right),
\end{align*}
and
\pagebreak
\begin{align*}
S_{3,j}(t)
& = \left[\Gamma^{(V)^{m+1}}(f^{[m]}) \right]\left(\tilde{U}^{(l_1)}(t),\ldots,\tilde{U}^{(l_m)}(t)\right) - \left[\Gamma^{(V)^{m+1}}(f^{[m]}) \right]\left(\tilde{U}^{(l_1)}(0),\ldots,\tilde{U}^{(l_m)}(0)\right) \\
& + \left[\Gamma^{(V)^{m+1}}(f^{[m]})\right]\left(\widetilde{U_j}^{(l_1)}(0),\ldots,\widetilde{U_j}^{(l_m)}(t)\right) - \left[\Gamma^{(V)^{m+1}}(f^{[m]})\right]\left(\widetilde{U_j}^{(l_1)}(t),\ldots,\widetilde{U_j}^{(l_m)}(t)\right).
\end{align*}

First, note that $\tilde{U}$ is $n$-times differentiable and $1\leq l_k \leq n-1$, so the derivatives $\tilde{U}^{(l_k)}$ are $\mathcal{S}^p$-bounded in a neighborhood of $0$. Hence, according to Corollary \ref{PerturbationFormula1} and Lemma \ref{intermlemma} $(1)$, there exist a constant $c_1$, an integer $J_1$ and $\alpha>0$ such that
\begin{equation}\label{estimate7}
\forall \, j\geq J_1, \ \forall \ |t| < \alpha, \ \| S_1(t) \|_p \leq c_1 \left\| e^{i\mathcal{A}_j(t)}V - e^{i\mathcal{A}(t)}V \right\| \leq c_1 \epsilon |t|.
\end{equation}

Next, according to Proposition \ref{PerturbationFormula2}, we have
\begin{align*}
& S_{2,j}(t) \\
& = t \sum_{q=1}^{m+1} \left( \left[\Gamma^{(e^{i\mathcal{A}_j(t)}V)^{q},(V)^{m-q+1}}(f^{[m+1]})\right]\left(\tilde{U}^{(l_1)}(t),\ldots,\tilde{U}^{(l_{q-1})}(t), \frac{e^{i\mathcal{A}_j(t)}V - V}{t}, \tilde{U}^{(l_{q})}(t) , \ldots, \tilde{U}^{(l_m)}(t) \right) \right. \\
&\hspace*{0.5cm} - \left. \left[\Gamma^{(e^{i\mathcal{A}_j(t)}V)^{q},(V)^{m-q+1}}(f^{[m+1]})\right]\left(\widetilde{U_j}^{(l_1)}(t),\ldots,\widetilde{U_j}^{(l_{q-1})}(t), \frac{e^{i\mathcal{A}_j(t)}V - V}{t}, \widetilde{U_j}^{(l_{q})}(t) , \ldots, \widetilde{U_j}^{(l_m)}(t) \right) \right).
\end{align*}
According to Lemma \ref{intermlemma} $(2)$ and $(4)$, there exist $\beta>0$ and an integer $J_2\in \mathbb{N}$, such that, for every $1\leq k\leq m$, for every $j\geq J_2$ and every $|t| < \beta$,
$$
\| \tilde{U}^{(l_k)}(t) - \widetilde{U_j}^{(l_k)}(t) \|_p \leq \epsilon \quad \text{and} \quad \left\| \frac{e^{i\mathcal{A}_j(t)}V - V}{t} \right\|_p \ \text{is bounded}.
$$
Since $\tilde{U}^{(l_k)}(t)$ and $\widetilde{U_j}^{(l_k)}(t)$, $1\leq k \leq m$, are locally bounded around $0$, and using the fact that the operators $\Gamma^{(e^{i\mathcal{A}_j(t)}V)^{q},(V)^{m-q+1}}(f^{[m+1]})$ are uniformly bounded with respect to $t$, there exists a constant $C$ depending on $p$, $m$, $f$ and $U$ such that
\begin{equation}\label{estimate8}
\forall j\geq J_2, \, \forall |t| < \beta, \ \| S_{2,j}(t) \|_p \leq |t| \sum_{i=1}^{m+1} C \epsilon =: c_2 \epsilon |t|.
\end{equation}
Finally, to estimate $S_{3,j}(t)$, let us write, according to Lemma \ref{intermlemma} $(5)$,
$$
\tilde{U}^{(l_k)}(t) = \tilde{U}^{(l_k)}(0) + t R_{l_k}(t) \quad \text{and} \quad \widetilde{U_j}^{(l_k)}(t) = \widetilde{U_j}^{(l_k)}(0) + t R^j_{{l_k}}(t).
$$
It follows that
\begin{align*}
& \left[\Gamma^{(V)^{m+1}}(f^{[m]}) \right]\left(\tilde{U}^{(l_1)}(t),\ldots,\tilde{U}^{(l_m)}(t)\right) - \left[\Gamma^{(V)^{m+1}}(f^{[m]}) \right]\left(\tilde{U}^{(l_1)}(0),\ldots,\tilde{U}^{(l_m)}(0)\right) \\
& = \sum_{\substack{A_i(t) = \tilde{U}^{(l_i)}(0) \ \text{or} \ t R_{l_i}(t) \\ \exists \, 1 \leq i \leq m, \ A_i(t) = t R_{l_i}(t)}} \left[\Gamma^{(V)^{m+1}}(f^{[m]}) \right]\left(A_1(t),\ldots, A_m(t)\right)
\end{align*}
and, similarly,
\begin{align*}
& \left[\Gamma^{(V)^{m+1}}(f^{[m]})\right]\left(\widetilde{U_j}^{(l_1)}(0),\ldots,\widetilde{U_j}^{(l_m)}(t)\right) - \left[\Gamma^{(V)^{m+1}}(f^{[m]})\right]\left(\widetilde{U_j}^{(l_1)}(t),\ldots,\widetilde{U_j}^{(l_m)}(t)\right) \\
& = \sum_{\substack{B_i^j(t) = \widetilde{U_j}^{(l_i)}(0) \ \text{or} \ t R^j_{{l_i}}(t) \\ \exists \, 1 \leq i \leq m, \ B_i^j(t) = t R^j_{{l_i}}(t)}} \left[\Gamma^{(V)^{m+1}}(f^{[m]}) \right]\left(B_1^j(t),\ldots, B_m^j(t)\right).
\end{align*}
Hence, we only have to estimate the terms
$$
\left[\Gamma^{(V)^{m+1}}(f^{[m]}) \right]\left(A_1(t),\ldots, A_m(t)\right) - \left[\Gamma^{(V)^{m+1}}(f^{[m]}) \right]\left(B_1^j(t),\ldots, B_m^j(t)\right)
$$
where $A_i(t)=\tilde{U}^{(l_i)}(0)$ if and only if $B_i^j(t) = \widetilde{U_j}^{(l_1)}(0)$. Moreover, to simplify the notations, we assume that $A_1(t) = t R_{{l_1}}(t)$ and $B_1^j(t) = t R^j_{{l_1}}(t)$. In that case,
\begin{align*}
& \left[\Gamma^{(V)^{m+1}}(f^{[m]}) \right]\left(A_1(t),\ldots, A_m(t)\right) - \left[\Gamma^{(V)^{m+1}}(f^{[m]}) \right]\left(B_1(t),\ldots, B_m(t)\right) \\
& = t \left( \left[\Gamma^{(V)^{m+1}}(f^{[m]}) \right]\left(R_{{l_1}}(t), A_2(t), \ldots, A_m(t)\right) - \left[\Gamma^{(V)^{m+1}}(f^{[m]}) \right]\left(R^j_{{l_1}}(t), B_2^j(t),\ldots, B_m^j(t)\right) \right).
\end{align*}
According to Lemma \ref{intermlemma} $(4)$ and $(5)$, there exists an integer $J_3'$ and $\gamma'>0$ such that, for every $2\leq i \leq m$,
$$
\forall \, j\geq J_3', \, \forall |t| < \gamma' , \ \| R_{{l_1}}(t) - R^j_{{l_1}}(t)\| \leq \epsilon \quad \text{and} \quad \| A_i(t) - B_i^j(t) \|_p \leq \epsilon.
$$
It follows that there exists a constant $C'>0$ such that
$$
\left\|
\left[\Gamma^{(V)^{m+1}}(f^{[m]}) \right]\left(A_1(t),\ldots, A_m(t)\right) - \left[\Gamma^{(V)^{m+1}}(f^{[m]}) \right]\left(B_1^j(t),\ldots, B_m^j(t)\right) \right\|_p \leq C'  \epsilon |t|.
$$
In particular, there exist an integer $J^3$, $\gamma>0$ and a constant $c_3 > 0$ such that
\begin{equation}\label{estimate9}
\forall \, j\geq J_3, \, \forall |t| < \gamma, \ \|S_{3,j}(t)\|_p \leq c_3 \epsilon |t|.
\end{equation}
Setting $ c = c_1+c_2+c_3$, $\delta = \min \{ \alpha, \beta, \gamma \}$ and $J:= \max \{ J_1, J_2, J_3 \}$, we get from \eqref{estimate7}, \eqref{estimate8} and \eqref{estimate9} that
\begin{equation}\label{estimate10}
\forall \, j\geq J, \, \forall |t| < \delta, \ \| R_j(t) \|_p \leq c \epsilon |t|.
\end{equation}
From the expression of $L_j(t)$, it follows that there exist $J'\in \mathbb{N}$, $\delta'>0$ and a constant $c' > 0$ such that
\begin{align}\label{estimate4}
\forall \, j\geq J', \, \forall |t| < \delta', \ \| L_j(t) \|_p \leq c' \epsilon |t|.
\end{align}

\noindent \underline{Conclusion}. Fix an integer $j_0 \geq \max \{J_0, J'  \}$. According to \eqref{diffphim}, there exists $\delta''>0$ such that
$$
\forall \, |t| < \delta'', \ \| \varphi_{j_0}^{(n-1)}(t) - \varphi_{j_0}^{(n-1)}(0) - t\varphi_{j_0}^{(n)}(0) \|_p \leq \epsilon |t|.
$$
According to \eqref{estimate3} and \eqref{estimate4}, we get from the equality \eqref{reformulate1} that, for every $t\in I$ such that $|t| < \min \{ \delta', \delta'' \}$,
\begin{align*}
\|\varphi^{(n-1)}(t) - \varphi^{(n-1)}(0) - tT\|_p \leq (c'+c_0+1) \epsilon |t|.
\end{align*}
Hence, we proved that
$$
\varphi^{(n-1)}(t) - \varphi^{(n-1)}(0) - tT = o(t),
$$
which shows that $\varphi^{(n-1)}$ is differentiable at $t=0$ with $\varphi^{(n)}(0) = T$, and finishes the proof.
\end{proof}

\pagebreak

We conclude this paper by proving Lemma \ref{intermlemma}.

\begin{proof}[Proof of Lemma \ref{intermlemma}] Throughout this proof, we fix $\epsilon > 0$.\\
- To prove $(1)$, note that by Duhamel's formula (see, e.g., \cite[Lemma 5.2]{AzCaDoSu09}), we have
\begin{align*}
\left\| e^{i\mathcal{A}(t)}V - e^{i\mathcal{A}_j(t)}V \right\|_p = \left\| e^{i\mathcal{A}(t)} - e^{i\mathcal{A}_j(t)} \right\|_p \leq \| \mathcal{A}(t) - \mathcal{A}_j(t) \|_p.
\end{align*}
Recall that $\mathcal{A}(0)=0$, so we can write $\mathcal{A}(t)=t\mathcal{A}'(0)+o(t)$ as $t\to 0$, and hence $\mathcal{A}_j(t) = tP_j \mathcal{A}'(0)P_j + P_jo(t) P_j$. It follows that
\begin{align*}
\| \mathcal{A}(t) - \mathcal{A}_j(t) \|_p \leq |t| \| \mathcal{A}'(0) - P_j \mathcal{A}'(0) P_j \|_p + \| o(t) - P_j o(t) P_j\|_p.
\end{align*}
Since $ \mathcal{A}'(0) \in \mathcal{S}^p(\hilh)$, for $j$ large enough, we have
$$
\| \mathcal{A}'(0) - P_j \mathcal{A}'(0) P_j \|_p \leq \epsilon,
$$
and for $|t|$ small enough, we have
$$
\| o(t) - P_j o(t) P_j\|_p \leq 2 \|o(t)\|_p \leq \epsilon |t|,
$$
which gives the desired inequality.

- The proof of $(2)$ is similar. Indeed, it suffices to write
\begin{align*}
\left\| e^{i\mathcal{A}(t)}V - V \right\|_p = \left\| e^{i\mathcal{A}(t)} - e^0 \right\|_p \leq \| \mathcal{A}(t) \|_p = |t| \left\| \frac{\mathcal{A}(t)}{t} \right\|_p \leq C|t|,
\end{align*}
where $C:=2\| \mathcal{A}'(0) \|_p$, for $t$ small enough. The proof of the second inequality is identical.

- For the rest of the proof, we let $g : t\mapsto e^{it}$. Then, we can write
$$
\tilde{U}(t) = \left[g(\mathcal{A}(t)) - g(\mathcal{A}(0))\right]V \quad \text{and} \quad \widetilde{U_j}(t) = \left[g(\mathcal{A}_j(t)) - g(\mathcal{A}_j(0))\right]VP_j.
$$
Since $g \in C^{\infty}(\mathbb{R})$ with bounded derivatives, by Corollary \ref{FormulaSA3}, $\tilde{U}$ and $\widetilde{U_j}$ are $n$-times differentiable on $\mathbb{R}$ and for every $1\leq k \leq n$ and every $t\in \mathbb{R}$,
\begin{align*}
\tilde{U}^{(k)}(t) = \left(\sum_{m=1}^{k} \sum_{\substack{l_1,\ldots,l_m\ge 1 \\ l_1+\cdots+l_m=k}}\dfrac{k!}{l_1!\cdots l_m!} \underbrace{\left[\Gamma^{(\mathcal{A}(t))^{m+1}}(g^{[m]})\right]\left(\mathcal{A}^{(l_1)}(t),\ldots,\mathcal{A}^{(l_m)}(t)\right)}_{:=D_{l_1,\ldots,l_m}(t)} \right) V,
\end{align*}
and
\begin{align*}
\widetilde{U_j}^{(k)}(t)
& = \left(\sum_{m=1}^{k} \sum_{\substack{l_1,\ldots,l_m\ge 1 \\ l_1+\cdots+l_m=k}}\dfrac{k!}{l_1!\cdots l_m!}\left[\Gamma^{(\mathcal{A}_j(t))^{m+1}}(g^{[m]})\right]\left((\mathcal{A}_j)^{(l_1)}(t),\ldots,(\mathcal{A}_j)^{(l_m)}(t)\right) \right) VP_j \\
& = \left(\sum_{m=1}^{k} \sum_{\substack{l_1,\ldots,l_m\ge 1 \\ l_1+\cdots+l_m=k}}\dfrac{k!}{l_1!\cdots l_m!} \underbrace{\left[\Gamma^{(\mathcal{A}_j(t))^{m+1}}(g^{[m]})\right]\left(P_j\mathcal{A}^{(l_1)}(t)P_j,\ldots,P_j\mathcal{A}^{(l_m)}(t)P_j\right)}_{:=D^j_{l_1,\ldots,l_m}(t)}\right) VP_j.
\end{align*}
To complete the proof of $(3)$, note that for every $t\in \mathbb{R}$,
$$
P_j\widetilde{U_j}(t)P_j = P_j(e^{i\mathcal{A}_j(t)}VP_j - VP_j)P_j = e^{i\mathcal{A}_j(t)}VP_j - VP_j = \widetilde{U_j}(t),
$$
which follows from the fact that $P_j$ commute with $e^{i\mathcal{A}_j(t)}$ and $V$. Hence, differentiating this formula $k$ times gives the result.

- Next, according to the latter, to prove $(4)$, we only have to estimate
$$
\| D_{l_1,\ldots,l_m}(t)V - D^j_{l_1,\ldots,l_m}(t)VP_j \|_p
$$
for some $1\leq m \leq n-1$ and $l_1, \ldots, l_m \geq 1$ such that $l_1+\cdots+l_m \leq n-1$. But it is easy to check that
$$D^j_{l_1,\ldots,l_m}(t)VP_j = D^j_{l_1,\ldots,l_m}(t)P_j V = D^j_{l_1,\ldots,l_m}(t)V,$$
so that
\begin{equation}\label{Lemmeetapesimpl}
\| D_{l_1,\ldots,l_m}(t)V - D^j_{l_1,\ldots,l_m}(t)VP_j \|_p = \| D_{l_1,\ldots,l_m}(t) - D^j_{l_1,\ldots,l_m}(t) \|_p.
\end{equation}
Denote $T_t:=\left[\Gamma^{(\mathcal{A}(t))^{m+1}}(g^{[m]})\right]$ and $T_{t,j} := \left[\Gamma^{(\mathcal{A}_j(t))^{m+1}}(g^{[m]})\right]$. We have
\begin{align*}
& D_{l_1,\ldots,l_m}(t) - D^j_{l_1,\ldots,l_m}(t) \\
& \hspace*{0,5cm} = \left( T_t\left(\mathcal{A}^{(l_1)}(t),\ldots,\mathcal{A}^{(l_m)}(t)\right) - T_t\left(\mathcal{A}^{(l_1)}(0),\ldots,\mathcal{A}^{(l_m)}(0)\right) \right) \\
& \hspace*{1cm} + \left( T_t\left(\mathcal{A}^{(l_1)}(0),\ldots,\mathcal{A}^{(l_m)}(0)\right) - T_t\left(P_j\mathcal{A}^{(l_1)}(0)P_j,\ldots,P_j\mathcal{A}^{(l_m)}(0)P_j\right)  \right) \\
& \hspace*{1cm} + \left( T_t\left(P_j\mathcal{A}^{(l_1)}(0)P_j,\ldots,P_j\mathcal{A}^{(l_m)}(0)P_j\right)  - T_{t,j}\left(P_j\mathcal{A}^{(l_1)}(0)P_j,\ldots,P_j\mathcal{A}^{(l_m)}(0)P_j\right)  \right) \\
& \hspace*{1cm} + \left( T_{t,j}\left(P_j\mathcal{A}^{(l_1)}(0)P_j,\ldots,P_j\mathcal{A}^{(l_m)}(0)P_j\right)  - T_{t,j}\left(P_j\mathcal{A}^{(l_1)}(t)P_j,\ldots,P_j\mathcal{A}^{(l_m)}(t)P_j\right)  \right) \\
& \hspace*{1cm} := K_1(t)+K_{2,j}(t)+K_{3,j}(t)+K_{4,j}(t).
\end{align*}
The continuity of $\mathcal{A}^{(l_k)}$ at $t=0$, $1\leq k \leq m$, and the uniform boundedness of $(T_t)_{t\in \mathbb{R}}$ give the existence of $C_1$ (depending on $f$, $\mathcal{A}$ and $p$) and $\alpha_1>0$ such that
$$
\forall \, |t| < \alpha_1, \ \|K_1(t)\|_p \leq C_1\max_{1\leq k \leq m} \|\mathcal{A}^{(l_k)}(t) - \mathcal{A}^{(l_k)}(0) \|_p \leq \epsilon.
$$
To estimate $K_{2,j}(t)$, it is enough to notice that since $\mathcal{A}^{(l_k)}(0) \in \mathcal{S}^p(\hilh)$, $1\leq k \leq m$,
$$
P_j\mathcal{A}^{(l_k)}(0)P_j \to \mathcal{A}^{(l_k)}(0),
$$
in $\mathcal{S}^p$ as $j\to +\infty$, so that
$$
\|K_{2,j}(t)\| \leq C_1 \max_{1\leq k \leq m} \| \mathcal{A}^{(l_k)}(0) - P_j\mathcal{A}^{(l_k)}(0)P_j \|_p \leq \epsilon
$$
for $j\geq J$ large enough. For the third term, there exists, by Corollary \ref{PerturbationFormula1}, a constant $C_2$ (depending on $f$ and $p$) such that
\begin{align*}
\|K_{3,j}(t) \|_p \leq C_2 \| \mathcal{A}(t) - \mathcal{A}_j(t) \|_p \|P_j\mathcal{A}^{(l_1)}(0)P_j\| \cdots \|P_j\mathcal{A}^{(l_m)}(0)P_j \|_p
& \leq \epsilon
\end{align*}
for $j$ large enough and $|t| < \alpha_2$ small enough, according to the proof of $(1)$. Since $K_{4,j}(t)$ can be estimated like $K_1(t)$, this concludes the proof of $(4)$.

- Finally, to prove $(5)$, write, for $0\leq k \leq n-1$ and $t\neq 0$,
$$
R_{k}(t) := \frac{\tilde{U}^{(k)}(t) - \tilde{U}^{(k)}(0)}{t} \quad \text{and} \quad R^j_{{k}}(t) := \frac{\widetilde{U_j}^{(k)}(t) - \widetilde{U_j}^{(k)}(0)}{t},
$$
so that
$$
\tilde{U}^{(k)}(t) = \tilde{U}^{(k)}(0) + t R_{k}(t) \quad \text{and} \quad \widetilde{U_j}^{(k)}(t) = \widetilde{U_j}^{(k)}(0) + t R^j_{{k}}(t).
$$
Using the same notations as before, it follows from the expressions of $\tilde{U}^{(k)}(t)$ and $\widetilde{U_j}^{(k)}(t)$ and from \eqref{Lemmeetapesimpl}, that to estimate
$$\| R_{k}(t) - R^j_{{k}}(t) \|_p,$$
it suffices to estimate the quantity 
$$
\frac{1}{|t|}\| (D_{l_1,\ldots,l_m}(t) - D_{l_1,\ldots,l_m}(0)) - (D^j_{l_1,\ldots,l_m}(t) - D^j_{l_1,\ldots,l_m}(0)) \|_p,
$$
for some $1\leq m \leq n-1$ and $l_1, \ldots, l_m \geq 1$ such that $l_1+\cdots+l_m \leq n-1$. To do so, let us write, with the notations $T_t$ and $T_{t,j}$ introduced above,
\begin{align*}
& \left((D_{l_1,\ldots,l_m}(t) - D_{l_1,\ldots,l_m}(0)) - (D^j_{l_1,\ldots,l_m}(t) - D^j_{l_1,\ldots,l_m}(0)) \right)\\
& = \left((T_t-T_{t,j})(P_j\mathcal{A}^{(l_1)}(t)P_j,\ldots,P_j\mathcal{A}^{(l_m)}(t)P_j)\right)\\
& \hspace*{0,5cm} + \left((T_t-T_0)(\mathcal{A}^{(l_1)}(t),\ldots,\mathcal{A}^{(l_m)}(t)) - (T_t-T_0)(P_j\mathcal{A}^{(l_1)}(t)P_j,\ldots,P_j\mathcal{A}^{(l_m)}(t)P_j)\right)\\
& \hspace*{0,5cm} + \left[T_0(P_j\mathcal{A}^{(l_1)}(0)P_j,\ldots,P_j\mathcal{A}^{(l_m)}(0)P_j) - T_0(P_j\mathcal{A}^{(l_1)}(t)P_j,\ldots,P_j\mathcal{A}^{(l_m)}(t)P_j)\right. \\
& \hspace*{1cm} + \left. T_0(\mathcal{A}^{(l_1)}(t),\ldots,\mathcal{A}^{(l_m)}(t)) - T_0(\mathcal{A}^{(l_1)}(0),\ldots,\mathcal{A}^{(l_m)}(0))\right]
\end{align*}
Denote by $L_1^j(t)$ and $L_2^j(t)$ the quantities on the first two lines in the last equality, and by $L_3^j(t)$ the quantity on the last two lines.

For $L_1^j$, by the boundedness of $\mathcal{A}^{(l_k)}(t)$, $1\leq k \leq m$, in a neighborhood of $0$ and by Corollary \ref{PerturbationFormula1}, we have the existence of $D_1>0$ such that
$$
\| L_1^j(t) \|_p \leq D_1 \| \mathcal{A}_j(t) - \mathcal{A}(t) \|_p \leq \epsilon |t|
$$
for $j$ large enough and $|t|$ small enough, according to the item $(1)$.

For the term $L_2^j$, Corollary \ref{PerturbationFormula1} gives the existence of $D_2>0$ such that
$$
\|T_t-T_0|\|_{\mathcal{B}_m(\mathcal{S}^p(\hilh)} \leq D_2 \| \mathcal{A}(t)\|_p \leq D_2' |t|
$$
for $t$ small enough and for some constant $D_2'$. Using again the boundedness $\mathcal{A}^{(l_k)}(t)$, $1\leq k \leq m$, in a neighborhood of $0$, we get the existence of $D_2''$ such that
$$
\| L_2^j(t) \|_p \leq D_2'' |t| \max_{1\leq k \leq m} \| \mathcal{A}^{(l_k)}(t) - P_j\mathcal{A}^{(l_k)}(t)P_j \|_p \leq \epsilon |t|,
$$
where the last inequality is obtained by writing $\mathcal{A}^{(l_k)}(t) =  \mathcal{A}^{(l_k)}(0) + t \mathcal{A}^{(l_k+1)}(0) + o(t)$ and applying the same computations as in the item $(1)$.

 Finally, let us write, for each $1\leq k\leq m$,
 $$
  \mathcal{A}^{(l_k)}(t) =  \mathcal{A}^{(l_k)}(0) + t \frac{ \mathcal{A}^{(l_k)}(t) -  \mathcal{A}^{(l_k)}(0)}{t} =: \mathcal{A}^{(l_k)}(0) + t G_k(t),
 $$
and
$$
P_j\mathcal{A}^{(l_k)}(t)P_j = P_j\mathcal{A}^{(l_k)}(0)P_j + t P_jG_k(t)P_j.
$$
We have
\begin{align*}
\| G_k(t) - P_jG_k(t)P_j \|_p
& \leq \left\| \frac{ \mathcal{A}^{(l_k)}(t) -  \mathcal{A}^{(l_k)}(0)}{t} - \mathcal{A}^{(l_k+1)}(0) \right\|_p + \| \mathcal{A}^{(l_k+1)}(0) - P_j \mathcal{A}^{(l_k+1)}(0) P_j \|_p \\
& \hspace*{0,5cm} + \left\| P_j\frac{\mathcal{A}^{(l_k)}(t) -  \mathcal{A}^{(l_k)}(0)}{t}P_j - P_j\mathcal{A}^{(l_k+1)}(0)P_j \right\|_p \\
& \leq 2 \left\| \frac{ \mathcal{A}^{(l_k)}(t) -  \mathcal{A}^{(l_k)}(0)}{t} - \mathcal{A}^{(l_k+1)}(0) \right\|_p + \| \mathcal{A}^{(l_k+1)}(0) - P_j \mathcal{A}^{(l_k+1)}(0) P_j \|_p\\
& \leq \epsilon,
\end{align*}
for $j$ large enough and $t$ small enough. Now, following the same computations used to estimate the term $S_{3,j}(t)$ in the proof of Theorem \ref{main}, we obtain, taking larger $j$ and smaller $|t|$ if necessary, the estimate
$$
\| L_3^j(t)\| \leq \epsilon |t|.
$$
In particular, we proved that there exist $J\in \mathbb{N}$ and $\alpha>0$ such that
$$
\forall \, j\geq J, \, \forall \, |t| < \alpha, \ \| R_{k}(t) - R^j_{{k}}(t) \|_p \leq \epsilon.
$$
This concludes the proof of the Lemma.
\end{proof}

\end{document}